\documentclass[reqno]{amsart}
\usepackage{amsmath}
\usepackage{amsthm}
\usepackage{amssymb}

\usepackage[utf8]{inputenc}
\usepackage[T1]{fontenc}
\usepackage{lmodern}
\usepackage{eucal}
\pdfoutput=1 
\usepackage{microtype}
\usepackage{url}
\usepackage{enumitem}
\usepackage{tikz-cd}
\usepackage{bbm}
\usepackage{longtable,booktabs}

\swapnumbers

\usepackage{hyperref}

\usepackage{xcolor}
\hypersetup{
    colorlinks,
    linkcolor={red!50!black},
    citecolor={blue!50!black},
    urlcolor={blue!80!black}
}

\usepackage[capitalise,nameinlink]{cleveref}
\crefname{subsection}{Subsection}{Subsections}
\crefname{subsubsection}{Subsubsection}{Subsubsections}

\theoremstyle{definition}
\newtheorem{theorem}{Theorem}[subsection]
\newtheorem{defn}[theorem]{Definition}

\newtheorem{ex}[theorem]{Example}
\newtheorem{cor}[theorem]{Corollary}
\newtheorem{lemma}[theorem]{Lemma}
\newtheorem{prop}[theorem]{Proposition}
\newtheorem{rmk}[theorem]{Remark}

\newtheorem{question}[theorem]{Question}

\newtheorem{problem}[theorem]{Problem}
\newtheorem*{rmk*}{Remark}
\newtheorem*{ex*}{Example}
\newtheorem*{theorem*}{Theorem}
\newtheorem*{defn*}{Definition}

\newcommand{\bbZ}{\mathbb{Z}}
\newcommand{\bbG}{\mathbb{G}}
\newcommand{\bbH}{\mathbb{H}}
\newcommand{\bbF}{\mathbb{F}}

\newcommand{\bbR}{\mathbb{R}}

\newcommand{\bbC}{\mathbb{C}}

\newcommand{\calF}{\mathcal{F}}

\newcommand{\Sp}{\mathcal{S}\mathrm{p}}
\newcommand{\Mod}{\mathcal{M}\mathrm{od}}

\newcommand{\Mack}{\mathcal{M}\mathrm{ack}}

\newcommand{\Map}{\operatorname{Map}}
\newcommand{\Hom}{\operatorname{Hom}}
\newcommand{\Aut}{\operatorname{Aut}}

\newcommand{\Fib}{\operatorname{Fib}}

\newcommand{\coker}{\operatorname{Coker}}
\newcommand{\colim}{\operatorname*{colim}}

\newcommand{\im}{\operatorname{Im}}
\renewcommand{\ker}{\operatorname{Ker}}

\newcommand{\res}{\operatorname{res}}
\newcommand{\tr}{\operatorname{tr}}

\newcommand{\Th}{\operatorname{Th}}

\newcommand{\End}{\operatorname{End}}
\newcommand{\ord}{\operatorname{ord}}
\newcommand{\Pic}{\operatorname{Pic}}

\newcommand{\Spin}{\operatorname{Spin}}

\newcommand{\h}{\mathrm{h}}
\newcommand{\op}{\mathrm{op}}

\newcommand{\cl}{\mathrm{cl}}

\newcommand{\alg}{\mathrm{alg}}

\newcommand{\bs}{{-}}
\newcommand{\bbs}{{=}} 
\newcommand{\ul}{\underline}
\newcommand{\ol}{\overline}

\newcommand\xqed[1]{%
  \leavevmode\unskip\penalty9999 \hbox{}\nobreak\hfill
  \quad\hbox{#1}}
\newcommand\tqed{\xqed{$\triangleleft$}}

\makeatletter
\DeclareRobustCommand{\tvdots}{%
  \vbox{\baselineskip4\p@\lineskiplimit\z@\kern0\p@\hbox{.}\hbox{.}\hbox{.}}}
\makeatother

\newcommand{\rightsim}{\xrightarrow{\raisebox{-1pt}{\tiny{$\sim$}}}}

\makeatletter
\@namedef{subjclassname@2020}{\textup{2020} Mathematics Subject Classification}
\makeatother

\begin{document}

\title[Equivariant equivalences of the form $\Sigma^V X \simeq \Sigma^W X$]{Equivalences of the form $\Sigma^V X \simeq \Sigma^W X$ \\
in equivariant stable homotopy theory}
\author[William Balderrama]{William Balderrama}


\begin{abstract}
We study equivalences of the form $\Sigma^{V}X\simeq \Sigma^{W}X$, where $G$ is a compact Lie group, $X$ is a $G$-spectrum, and $V$ and $W$ are $G$-representations. These equivalences encode a periodicity phenomenon in $G$-equivariant homotopy theory which generalizes the classical James periodicity for $G = C_2$. 

When $X = C(a_\lambda)$ is the cofiber of an Euler class, we construct an $RO(G)$-graded $J$-homomorphism $J\colon \pi_\lambda KO_G\rightarrow \pi_\star^G C(a_\lambda)^\times$ which gives control over these periodicities. It also produces infinite periodic families in the $G$-equivariant stable stems. We illustrate this with several explicit examples.

More generally, our work gives information about $RO(G)$-graded units in equivariant stable cohomotopy rings. We apply this to construct universal periodicities and differentials in the $G$-homotopy fixed point spectral sequence, and other equivariant Atiyah--Hirzebruch spectral sequences.
\end{abstract}

\maketitle

\vspace{-0.55\baselineskip}
\section{Introduction}
\subsection{Background}\label{ssec:background}

In \cite{bredon1967equivariant,bredon1968equivariant}, Bredon introduced the $C_2$-equivariant stable stems. In modern notation, these are the groups $\pi_\star S_{C_2}$ comprising the $RO(C_2)$-graded coefficient ring of the $C_2$-equivariant sphere spectrum. At the same time, he introduced groups that, in modern notation, may be identified as the $RO(C_2)$-graded homotopy groups $\pi_\star^{C_2} C(a_\sigma^{m})$ of the cofiber of the Euler class $a_\sigma^{m}\colon S^{-m\sigma}_{C_2}\rightarrow S^0_{C_2}$.

In addition to fitting these groups into the evident long exact sequences, Bredon observed that they satisfy a certain periodicity:
\begin{equation}\label{eq:bredon}
\pi_{\star+2^{\gamma(m)}}^{C_2} C(a_\sigma^{m+1})\cong \pi_{\star+2^{\gamma(m)}\sigma}^{C_2}C(a_\sigma^{m+1}),
\end{equation}
where $\gamma(m) = \#\{0<k\leq m : k\equiv 0,1,2,4\pmod{8}\}$.

This periodicity was further studied by Araki \cite{araki1979forgetful}.
The $C_2$-spectrum $C(a_\sigma^{m+1})$ is dual to
the $m$-sphere $\ul{S}^m$ with its free antipodal action. The theory of Clifford algebras provides a trivialization of the $C_2$-equivariant vector bundle $2^{\gamma(m)}\sigma\times \ul{S}^m \rightarrow \ul{S}^m$, and Araki used this to produce and study invertible elements
\begin{equation}\label{eq:uself}
\omega_m \in \pi_{2^{\gamma(m)}(1-\sigma)}^{C_2}C(a_\sigma^{m+1})^\times.
\end{equation}
Additional properties of these were established by Araki and Iriye in \cite{arakiiriye1982equivariant}, where they were applied to the computation of $C_2$-equivariant stable stems.

As already observed by Bredon, the groups $\pi_\star^{C_2} C(a_\sigma^{m+1})$ can be identified as nonequivariant stable homotopy groups of real stunted projective spaces. From this perspective, \cref{eq:bredon} is a consequence of \textit{James periodicity} \cite{james1958crosssections}, as also noted by Landweber \cite{landweber1968conjugations}. In \cite{behrensshah2020c2}, Behrens and Shah used the Adams isomorphism and James periodicity directly to produce invertible elements of the form \cref{eq:uself}. Their work emphasizes the interaction with the $C_2$-equivariant Adams spectral sequence, constructing these elements to lift powers of the orientation class $u_\sigma \in \pi_{1-\sigma}H\bbF_2^{C_2}$.

Our goal in this paper is to put these results in the context of a more general periodicity phenomenon in equivariant stable homotopy theory, and give applications.

\subsection{Summary}

The starting point of our investigation is a reinterpretation and generalization of these classical results in terms of the \textit{$J$-homomorphism}. Classically, the stable $J$-homomorphism for a space $Z$ is a homomorphism
\begin{equation}\label{eq:classicaljhom}
KO^{-1}(Z)\rightarrow \pi_0 D(\Sigma^\infty_+ Z)^\times,
\end{equation}
where $KO^{-1}(Z)$ is the $K$-theory of real vector bundles over the suspension of $Z$ and $\pi_0 D(\Sigma^\infty_+ Z)^\times$ is the group of units in the stable cohomotopy ring of $Z$. The $J$-homomorphism was originally introduced by G.W.\ Whitehead \cite{whitehead1942homotopy}, and the construction works just as well equivariantly, as has been studied by several people going back to Segal \cite{segal1971equivariant}.

We introduce the following refinement of the equivariant $J$-homomorphism. Fix a compact Lie group $G$ and compact $G$-space $Z$, with \textit{unreduced} suspension $SZ$.

\begin{theorem}[\cref{thm:jhom}]\label{thm:jhomomorphism}
The equivariant $J$-homomorphism refines to a homomorphism
\[
J\colon \widetilde{KO}{}_G^0(SZ)\rightarrow \pi_\star^G D(\Sigma^\infty_+ Z)^\times,
\]
defined up to signs, of the following signature: if $\xi \in \widetilde{KO}{}_G^0(SZ)$ restricts along the inclusion $S^0\rightarrow SZ$ to $\alpha \in \widetilde{KO}{}_G^0(S^0)\cong RO(G)$, then $J(\xi) \in \pi_\alpha^G D(\Sigma^\infty_+ Z)$.
\tqed
\end{theorem}

This refines the classical equivariant $J$-homomorphism as in \cref{eq:classicaljhom} in the sense that the latter is obtained by restricting along the canonical map $SZ\rightarrow\Sigma(Z_+)$. By ``defined up to signs'' we mean that $J$ takes values in orbits for the action by the subgroup of orthogonal units in $(\pi_0 S_G)^\times$, i.e.\ those obtained by compactifying an automorphism of $G$-representations. See \cref{rmk:signs} for further discussion.

\begin{ex}\label{ex:repsphere}
Let $Z = S(\lambda)$ be the unit sphere in a $G$-representation $\lambda$. In this case the cofiber sequence $S(\lambda)_+\rightarrow S^0\rightarrow S^\lambda$ yields equivalences
\[
SZ\simeq S^\lambda,\qquad D(\Sigma^\infty_+Z)\simeq C(a_\lambda),
\]
where $a_\lambda \in \pi_{-\lambda}S_G$ is the Euler class of $\lambda$, so our $J$-homomorphism takes the form
\[
J\colon \pi_\lambda KO_G\rightarrow \pi_\star^G C(a_\lambda)^\times,
\]
where if $\xi \in \pi_\lambda KO_G$ then
\[
|J(\xi)|=\alpha \qquad \text{when}\qquad a_\lambda\xi = \alpha \in \pi_0 KO_G = RO(G)\eqno\triangleleft
\]
\end{ex}

\begin{ex}\label{ex:james}
It follows from the structure of $\pi_0 KO_{C_2}\otimes C(a_\sigma^{m+1})\cong KO^0(\bbR P^m)$ \cite{adams1962vector} that
\[
\im(a_\sigma^{m+1}\colon \pi_{(m+1)\sigma}KO_{C_2}\rightarrow RO(C_2)) = \bbZ\{2^{\gamma(m)}(1-\sigma)\},
\]
so \cref{thm:jhomomorphism} encodes the classical James periodicity of \cref{ssec:background}.
\tqed
\end{ex}

\begin{rmk}\label{rmk:periodicity}
In general, an invertible element $u\in \pi_\alpha^G C(a_\lambda)$ gives rise to a secondary \textit{periodicity operator}
\[
P_u\colon\pi_\star^G(\bs)\rightharpoonup \pi_{\star+\alpha}^G(\bs),
\]
defined on the kernel of $a_\lambda$ and defined modulo the projection $\partial(u) \in \pi_{\alpha+\lambda-1}S_G$ of $u$ onto the top cell of $C(a_\lambda)$. These operators are periodic in the sense that if $a_\lambda x = 0$ then $x \in P_{u^{-1}}(P_u(x))$. Thus \cref{ex:repsphere} implies that the equivariant $K$-theory of representation spheres parametrizes certain nontrivial periodicities defined on Euler-torsion in $G$-equivariant stable homotopy theory.
\tqed
\end{rmk}

As \cref{rmk:periodicity} suggests, \cref{thm:jhomomorphism} can be used to produce \textit{infinite periodic families} in the $RO(G)$-graded equivariant stable stems.

\begin{ex}[\cref{ssec:quaternions}]\label{ex:quat}
Let $\bbH$ denote the tautological representation of the quaternion group $Q_8\subset\Sp(1)$ of order $8$. If we define
\[
p(n) = \begin{cases}
2n,&n\text{ even},\\
2n+1,&n\text{ odd},
\end{cases}
\]
then $2^{p(n)}(4-\bbH) = a_\bbH^{n+1} \cdot b_{n+1}$ for some $b_{n+1}\in \pi_{(n+1)\bbH}KO_{Q_8}$ which lifts a generator of $\pi_{4(n+1)}KO$ when $n$ is even and $4$ times a generator when $n$ is odd. From this one obtains invertible elements
\[
u_{2^{p(n)}\bbH} = J(b_{n+1}) \in \pi_{2^{p(n)}(4-\bbH)}^{Q_8} C(a_\bbH^{n+1}).
\]
By taking powers and projecting onto the top cell of $C(a_\bbH^{n+1})$, this gives for each $n$ the infinite periodic family
\[
\partial(u_{2^{p(n)}\bbH}^k) \in \pi_{2^{p(n)}k(4-\bbH)+(n+1)\bbH-1}S_{Q_8},
\]
all nonzero for $k\neq 0$, satisfying
\[
\res^{Q_8}_e(\partial(u_{2^{p(n)}}^k)) = \begin{cases}
k\cdot j_{4n+3},&n\text{ even},\\
k\cdot 4j_{4n+3},&n\text{ odd},
\end{cases}
\]
where $j_{n} \in \pi_{n}S$ is a generator of the image of $J$ in this degree.
\tqed
\end{ex}

In \cref{sec:ahss}, we apply \cref{thm:jhomomorphism} to construct universal periodicities and differentials in equivariant Atiyah--Hirzebruch spectral sequences. We refer the reader there for the general situation and just state here a special case.

Given a finite group $G$ and $G$-ring spectrum $R$, there is a \textit{homotopy fixed point spectral sequence}
\[
E_2(R) = H^\ast(G;\pi_\star^e R)\Rightarrow \pi_{\star-\ast}^GR_h^\wedge
\]
computing the $RO(G)$-graded homotopy of the Borel completion $R_h^\wedge = F(EG_+,R)$. This spectral sequence exhibits periodic behavior: there are invertible elements $u_\alpha \in \pi_{|\alpha|-\alpha}^e R\cong \pi_0^e R$ for each $\alpha\in RO(G)$, canonical up to a sign, allowing one to identify
\[
E_2(R) = H^\ast(G;\pi_\ast^e R[u_\lambda^{\pm 1} : \lambda\text{ irreducible $G$-representation}]).
\]
Here, $g\in G$ acts on $u_\lambda$ by $\pm 1$ according to whether the action of $g$ on $\lambda$ is oriented. \cref{thm:jhomomorphism} allows us to construct universal differentials on these classes. If $\bbZ\leftarrow \bbZ[G]\leftarrow \bbZ[G]\{I_1\}\leftarrow\bbZ[G]\{I_2\}\leftarrow \cdots$ is the free $\bbZ[G]$-module resolution associated to a cell structure on $EG$, then an $E_1$-page for the homotopy fixed point spectral sequence is given by
\[
E_1^{\alpha,f}(R) = \pi_{|\alpha|+f}^eR^{\times I_f}\Rightarrow\pi_\alpha^G R_h^\wedge.
\]
In the case $R = KO_G$ and in integer degrees, this is the classical Atiyah--Hirzebruch spectral sequence for $KO^\ast BG$, and the classical $J$-homomorphism defines a map
\[
\tilde{J}\colon E_1^{0,f}(KO_G) = \pi_fKO^{\times I_f}\rightarrow \pi_{f-1}S^{\times I_f} = E_1^{-1,f}(S_G).
\]
Given a $0$-dimensional virtual $G$-representation $\alpha$, set $t_\alpha = u_{-\alpha} \in  E_1^{\alpha,0}(R)$.

\begin{theorem}[\cref{thm:ahss}]\label{thm:hfpss}
Suppose that $\alpha\in RO(G)$ is detected in the Atiyah--Hirzebruch spectral sequence for $KO^0(BG)$ in positive filtration $f$ by $b\in E_1^{0,f}(KO)$. Then $t_\alpha \in E_1^{\alpha,0}(R)$ survives to the $E_f$-page, whereupon
\[
d_f(t_\alpha) = \pm\tilde{J}(b)  t_\alpha
\]
in $E_f^{\alpha-1,f}(R)$.
\tqed
\end{theorem}

\begin{ex}
When $G = C_2$, the universal space $EC_2 = S(\infty\sigma)$ admits a cell structure with one free cell in each degree, giving
\begin{align*}
E_1^{s,f}(KO_{C_2}) &= \pi_{s+f}KO \Rightarrow KO^{-s}(BC_2),\\
E_1^{\star,\ast}(S_{C_2}) &= \pi_\ast S[u_\sigma^{\pm 1},a_\sigma]\Rightarrow \pi_{\star}(S_{C_2})_h^\wedge.
\end{align*}
Here, $a_\sigma$ is the generator of $H^1(C_2;\pi_{1-\sigma}^eS_{C_2})\cong\bbZ/(2)$. Let $\rho(n)$ denote the $n$th Hurwitz--Radon number, i.e.\ if $\nu_2(n)=4a+b$ with $0\leq b \leq 3$ then $\rho(n) = 8a+2^b$. Then $k(1-\sigma)$ is detected in the Atiyah--Hirzebruch spectral sequence for $KO^0(BC_2)$ by a generator of $\pi_{\rho(k)}KO$, and it follows that 
\[
d_{\rho(k)}(u_\sigma^k) = j_{\rho(k)-1} \cdot a_\sigma^{\rho(k)} u_\sigma^{-\rho(k)}\cdot u_\sigma^k,
\]
where $j_n\in \pi_n S$ is a generator of the image of $J$ in degree, with the understanding that $j_0 = \pm 2$.  These can also be interpreted as primary differentials in the Atiyah--Hirzebruch spectral sequence for real projective space. This essentially goes back to Adams \cite{adams1962vector}, compare also \cite[Theorem 3.5]{arakiiriye1982equivariant} and \cite[Proposition 2.17]{brunermaymccluresteinberger1986hinfinity}.
\tqed
\end{ex}

\begin{ex}
\cref{thm:hfpss} only produces universal differentials, which can vanish in a given $G$-ring spectrum. For example, the homotopy fixed point spectral sequence for Real bordism $M\bbR$ takes the form
\[
E_1^{\star,\ast}(M\bbR) = \pi_\ast MU[u_\sigma^{\pm 1},a_\sigma]\Rightarrow \pi_\star^{C_2} M\bbR.
\]
\cref{thm:hfpss} implies $d_1(u_\sigma) = \pm 2 a_\sigma$ and $d_2(u_\sigma^2) = \eta a_\sigma^2$. As $\eta$ vanishes in $MU$, this implies that $u_\sigma^2$ must survive to (at least) the $E_3$-page for $M\bbR$. In fact one has
\[
d_3(u_\sigma^2) = a_\sigma^3\ol{v}_1 ,
\]
where $\ol{v}_1 \in \pi_{1+\sigma}M\bbR$ is detected by $u_\sigma^{-1}v_1$ \cite{hukriz2001real}. Given this, one can instead interpret \cref{thm:hfpss} as saying that $a_\sigma \ol{v}_1 \in \pi_1 M\bbR$ is the Hurewicz image of a class in $\pi_1 S_{C_2}$ which is killed by $a_\sigma^2$ and lifts $\eta$. The only such class is $\eta_\cl + a_\sigma^2\nu_{C_2}$, where $\eta_\cl\in \pi_1 S_{C_2}$ and $\nu_{C_2}\in \pi_{1+2\sigma}S_{C_2}$ are the nonequivariant and equivariant complex and quaternionic Hopf maps respectively. As $\nu_{C_2}$ has trivial Hurewicz image in $M\bbR$, this encodes the identity $a_\sigma \ol{v}_1 = \eta_\cl$. Similar considerations apply to higher differentials and other $G$-spectra.
\tqed
\end{ex}

\begin{ex}
When $G = Q_8$, as $\bbH$ is oriented the homotopy fixed point spectral sequence restricted to degrees $\ast+\ast\bbH$ takes the form
\[
E_2 = H^\ast(Q_8;\pi_\ast^e R)[u_\bbH^{\pm 1}]\Rightarrow \pi_{\ast+\ast\bbH}^{Q_8}R_h^\wedge,
\]
where $|u_\bbH| = 4-\bbH$. Using the fact that the unit sphere $S(n\bbH)$ is a $(4n-1)$-skeleton of $EQ_8$, it follows from \cref{thm:hfpss} and \cref{ex:quat} that there are differentials
\[
d_{4(n+1)}(u_\bbH^{2^{p(n)}}) = \begin{cases}
j_{4n+3}\cdot  a_\bbH^{n+1}u_\bbH^{-(n+1)} \cdot u_\bbH^{2^{p(n)}},&n\text{ even},\\
4j_{4n+3} \cdot a_\bbH^{n+1}u_\bbH^{-(n+1)} \cdot u_\bbH^{2^{p(n)}},&n\text{ odd}
\end{cases}
\]
for $n\geq 0$, where $a_\bbH$ is detected by the generator of $H^4(Q_8;\pi_{4-\bbH}^e S_{Q_8}) \cong \bbZ/(8)$.
\tqed
\end{ex}

In \cref{sec:bordism} we investigate an additional property of the invertible elements produced by \cref{thm:jhomomorphism} that makes precise a certain analogy between these equivariant ``James-type periodicities'' and $v_n$-periodicity in classical stable homotopy theory. This analogy was first highlighted by Behrens and Shah \cite{behrensshah2020c2} in their construction of $C_2$-equivariant $u_\sigma$-self maps. Given a $0$-dimensional virtual complex $G$-representation $\alpha$, there is a canonical invertible \textit{Thom class}
\[
t_\alpha \in \pi_{\alpha}MU_G,
\]
where $MU_G$ is tom Dieck's homotopical $G$-equivariant complex cobordism spectrum.

\begin{defn}[\cref{def:talpha}]
A \textit{$t_\alpha$-element} of order $n>0$ in a $G$-ring spectrum $R$ is an invertible element $t \in \pi_{n\alpha}^G R$ lifting $t_\alpha^n$ under the Hurewicz map $R\rightarrow R \otimes MU_G$.
\tqed
\end{defn}

When $\alpha = |V| - V$ one might call these \textit{$u_V$-elements}. Note that multiplication by a $t_\alpha$-element induces a \textit{$t_\alpha$-self map}: a self-map inducing multiplication by a power of $t_\alpha$ in $MU_G$-theory. The observation here is now the following.

\begin{prop}[\cref{prop:jt2}]\label{prop:jt}
The complex $J$-homomorphism
\[
J\colon \widetilde{KU}{}_G^0(SZ)\rightarrow\pi_\star^G D(\Sigma^\infty_+Z)^\times
\]
takes values in $t_\alpha$-elements for various $\alpha$: if $|J(\xi)|=\alpha$ then $J(\xi)$ is a $t_\alpha$-element.
\tqed
\end{prop}

The analogous proposition holds with $KU$ and $MU$ replaced by $KO$ and $MO$, or by $KSp$ and $MSp$.

\begin{ex}
As $H\bbF_2^{C_2}\otimes C(a_\sigma^{m+1})$ is $MO_{C_2}$-oriented, it follows from the unoriented analogue of \cref{prop:jt} that the invertible elements $u_{2^{\gamma(m)}\sigma} \in \pi_{2^{\gamma(m)}(1-\sigma)}^{C_2}C(a_\sigma^{m+1})$ guaranteed by \cref{ex:james} have Hurewicz image $u_\sigma^{2^{\gamma(m)}} \in (H\bbF_2^{C_2})_{2^{\gamma(m)}(1-\sigma)} C(a_\sigma^{m+1})$, where $u_\sigma \in \pi_{1-\sigma}H\bbF_2^{C_2}$ is the orientation class. In other words, we recover the $u_\sigma$-elements produced by Behrens--Shah in \cite[Theorem 7.7]{behrensshah2020c2}.
\tqed
\end{ex}

Using equivariant nilpotence techniques we are able to bootstrap \cref{prop:jt} into a general existence criterion for $t_\alpha$-elements.

\begin{theorem}[\cref{thm:maint}]
Let $G$ be a finite group and $R$ be a $G$-ring spectrum, and suppose
\[
\Phi^C R \neq 0 \implies \res^G_C\alpha = 0
\]
for all cyclic subgroups $C\subset G$. Then $R$ admits a $t_\alpha$-element of order dividing a power of $|G|$.
\tqed
\end{theorem}

There are more $RO(G)$-graded periodicities than are accounted for by $t_\alpha$-self maps. For example, there are stable equivalences $S^V\simeq S^W$ that do not come from an isomorphism $V\cong W$ of $G$-representations. \cref{sec:j} contains some observations about this situation, mostly adapting work of tom Dieck, Hauschild, Petrie, and Tornehave \cite{dieckpetrie1978geometric, dieck1979transformation, tornehave1982equivariant} on homotopy equivalent representation spheres and the equivariant Adams conjecture. For example, we prove the following.

\begin{theorem}[\cref{cor:jg}]\label{thm:generalequiv}
Let $G$ be a finite group and $X$ be a compact $G$-spectrum. Given $\alpha \in RO(G)$, there exists an equivalence $\Sigma^{n\alpha}X\simeq X$ for some $n\geq 1$ if and only if
\[
\Phi^C X \neq 0 \implies |\alpha^C| = 0
\]
for all cyclic subgroups $C\subset G$.
\tqed
\end{theorem}

\begin{ex}
Given a finite group $G$, $G$-representation $\lambda$, and $\alpha \in RO(G)$, there exists an equivalence $\Sigma^{n\alpha}C(a_\lambda)\simeq C(a_\lambda)$ for some $n\geq 1$ if and only if
\[
|\lambda^C|\neq 0 \implies |\alpha^C| = 0
\]
for all cyclic subgroups $C\subseteq G$. Thus it is a purely representation-theoretic condition that determines when $a_\lambda$-torsion is $\alpha$-periodic (of some possibly large period).
\tqed
\end{ex}

As a general statement this is quite satisfactory, but as a practical matter it gives no control over the equivalences $\Sigma^{n\alpha}X\simeq X$. More can be said after inverting some primes not dividing the order of $G$, and we make some comments about this situation in \cref{ssec:htpyequivreps}. A clean general statement is available when $G$ is a $p$-group, where we prove the following.

Fix a prime $p$ and finite $p$-group $G$. In this case, Bousfield localization $L_{KU_G/p}$ behaves similarly in $G$-equivariant homotopy theory to how $K(1)$-localization behaves in nonequivariant homotopy theory. In particular, if $\ell$ generates a dense subgroup of $\bbZ_p^\times/\{\pm 1\}$ and we set
\[
J_G = \Fib\left(\psi^\ell-\psi^1\colon (KO_G)_p^\wedge\rightarrow (KO_G)_p^\wedge\right),
\]
then $J_G\simeq L_{KU_G/p}S_G$, and more generally
\[
L_{KU_G/p}D(\Sigma^\infty_+Z)\simeq F(\Sigma^\infty_+Z,J_G)
\]
for any compact $G$-space $Z$ \cite[Corollary A.4.13]{balderrama2022total}. Write $j_{K(1)}^Z\colon RO(G)\rightarrow KO_G^0(Z)\rightarrow J_G^1(Z)$ for the resulting boundary map.

\begin{theorem}[\cref{thm:k1}]
Let $G$ be a finite $p$-group and $Z$ be a compact $G$-space. Given $\alpha \in RO(G)$, there exists an invertible element in $\pi_\alpha D(\Sigma^\infty_+ Z)_{(p)}$ if and only if $j_{K(1)}^Z(\alpha) = 0$.
\tqed
\end{theorem}

Thus the location of units in $\pi_\star^G D(\Sigma^\infty_+ Z)_{(p)}$ is completely determined by $K(1)$-local information.

In \cref{sec:examples} we work out several explicit examples. In particular we compute $\pi_{\ast\lambda}KO_G$ for a variety of finite groups $G$ and $G$-representations $\lambda$, producing explicit periodicities on $C(a_\lambda^n)$ and infinite periodic families in $\pi_\star S_G$.

\subsection*{Acknowledgements}

This work was supported by NSF RTG grant DMS-1839968.

\section{An \texorpdfstring{$RO(G)$}{RO(G)}-graded \texorpdfstring{$J$}{J}-homomorphism}\label{sec:jhomomorphism}

In this section we construct the equivariant $J$-homomorphism promised in the introduction. Throughout this section $G$ is a compact Lie group, and for simplicity we shall restrict our attention to compact $G$-spaces.

\subsection{Preliminaries}

We begin by fixing a bit of notation. Fix $F$ to be one of the real division algebras $\bbR$, $\bbC$, or $\bbH$. All vector spaces, vector bundles, and so forth are understood to be with respect to $F$. Write
\[
KF_G = \begin{cases} KO_G,&F=\bbR,\\ KU_G,&F = \bbC,\\ KSp_G,&F=\bbH\end{cases}
\]
for the $G$-spectrum representing the $K$-theory of $G$-equivariant vector bundles, and
\[
RF(G) =  \pi_0 KF_G
\]
for the corresponding Grothendieck group of $G$-representations.

If $Z$ is a compact $G$-space, then $KF_G^0(Z)$ can be identified explicitly as
\begin{align*}
KF_G^0(Z) = \{\text{$G$-equivariant vector bundles over $Z$}\} / (\sim),
\end{align*}
where
\[
\xi\sim\zeta\quad\text{when}\quad \xi\oplus V \cong \zeta\oplus V\text{ for some representation $V$}.
\]
The reduced $K$-group $\widetilde{KF}{}_G^0(Z)$, generally only defined when $Z$ is pointed, may be similarly identified as
\[
\widetilde{KF}{}_G^0(Z) = \{\text{$G$-equivariant vector bundles over $Z$}\} / (\approx),
\]
where
\[
\xi\approx\zeta\quad\text{when}\quad\xi\oplus V\cong \zeta\oplus W\text{ for some representations $V$ and $W$}.
\]
Write $SZ = S^0\ast Z$ for the unreduced suspension of $Z$, and $a_Z\colon S^0\rightarrow SZ$ for the inclusion of cone points. Then the restriction
\[
a_Z\colon \widetilde{KF}{}_G^0(SZ) \rightarrow \pi_0 KF_G \cong RF(G)
\]
sends a vector bundle $\xi$ over $SZ$ to to the difference $V-W$, where $V$ and $W$ are the restrictions of $\xi$ to the cone points $-1$ and $1$ respectively.

We separate out the following remark for easy reference.

\begin{rmk}\label{rmk:signs}
If $V$ is a $G$-representation, then the representation sphere $S^V$ is the pointed $G$-space defined as the one-point compactification of $V$. The suspension spectrum of $S^V$ is invertible as a $G$-spectrum, allowing one to construct virtual representation spheres $S^\alpha \in \Sp^G$ for any virtual representation $\alpha \in RO(G)$. However, whereas the assignment $V\mapsto S^V$ is natural in $V$, the $G$-spectrum $S^\alpha$ is generally defined only up to \textit{noncanonical} isomorphism. 

This issue arises from the fact that if $V$ is a $G$-representation, then $\Aut(V)$ may act nontrivially on $S^V$ in the homotopy category of $G$-spectra. If $V$ is defined over $F = \bbC$ or $\bbH$ and we consider only $F$-linear automorphisms, then in fact $\Aut(V)$ acts trivially on $S^V$ up to homotopy and no such issues arise. If $F = \bbR$ then we will sidestep this issue by taking the convention that, for our purposes, $S^\alpha$ and $\pi_\alpha(\bs)$ will only be \textit{defined} up to signs, where ``signs'' refers to the orthogonal units in $A(G)$, i.e.\ the image of tom Dieck's homomorphism
\[
j\colon RO(G)\rightarrow (\pi_0 S_G)^\times
\]
sending a $G$-representation $V$ to the class obtained by compactifying $-1\colon V\rightarrow V$. Unfortunately we see no way to avoid this while maintaining that the source of the $J$-homomorphism of \cref{thm:jhom} is $\widetilde{KO}{}_G^0(SZ)$, and not some more rigid but less uniformly computable object (such as $\Map^G(Z,L(V,W))$ for fixed $V$ and $W$).

We warn the reader that, having taken this convention, we will generally not take the care needed to pin down precise orientations and signs in intermediate arguments as they will not affect the final results.
\tqed
\end{rmk}

\subsection{Constructing the \texorpdfstring{$J$}{J}-homomorphism}

We now proceed to the construction. Fix a $G$-space $Z$ and vector bundle $\xi$ over $SZ$, and write $V$ and $W$ for the restrictions of $\xi$ to the cone points $-1$ and $1$. Let $L(V,W)$ denote the $G$-space of linear isometries $V\rightarrow W$, and write
\[
j\colon L(V,W)\rightarrow \Map_\ast(S^V,S^W)
\]
for the one-point compactification map. As $SZ$ is obtained by gluing two contractible spaces along a copy of $Z$, the vector bundle $\xi$ is classified by a clutching function
\[
\varphi_\xi\colon Z\rightarrow L(V,W),
\]
and the composite $j\circ\varphi_\xi\colon Z\rightarrow\Map_\ast(S^V,S^W)$ is adjoint to a map
\[
J(\xi)\colon \Sigma^V(Z_+)\rightarrow S^W.
\]
After stabilization, this defines a class
\[
J(\xi) \in \pi_{V-W}^GD(\Sigma^\infty_+Z)
\]
that we will denote by the same name.

\begin{theorem}\label{thm:jhom}
The above construction extends to a natural homomorphism
\[
J\colon \widetilde{KF}{}_G^0(SZ)\rightarrow \pi_\star^G D(\Sigma^\infty_+ Z)^\times,
\]
defined up to signs when $F = \bbR$, with the property that
\[
|J(\xi)| = \alpha \qquad\text{when}\qquad a_Z\xi = \alpha \in RF(G).
\]
\end{theorem}
\begin{proof}
If $Z_V$ is the trivial bundle $V\times Z \rightarrow Z$, then by construction $J(Z_V)\colon \Sigma^V(Z_+)\rightarrow S^V$ is the collapse map, adjoint to the unit $1 \in \pi_0 ^GD(\Sigma^\infty_+Z)$. Recall that if $\xi$ is any vector bundle over $SZ$, then there exists another vector bundle $\xi'$ for which $\xi\oplus\xi'\cong Z_V$ for some $G$-representation $V$ \cite[Proposition 2.4]{segal1968equivariant}. Therefore we may reduce to just verifying that $J(\xi\oplus \xi') = J(\xi)\cdot J(\xi')$ for any two vector bundles $\xi$ and $\xi'$ over $SZ$. Write
\[
\varphi_\xi\colon Z\rightarrow L(V,W),\qquad \varphi_{\xi'}\colon Z\rightarrow L(V',W')
\]
for clutching functions for $\xi$ and $\xi'$. Then a clutching function for $\xi\oplus\xi'$ is given by
\[
\varphi_{\xi\oplus\xi'}\colon Z\rightarrow L(V\oplus V',W\oplus W'),\qquad \varphi_{\xi\oplus\xi'}(z)(v,v') = (\varphi_\xi(z)(v),\varphi_{\xi'}(z)(v')).
\]
It follows that $J(\xi\oplus\xi')$ is the composite
\begin{center}\begin{tikzcd}[row sep=tiny]
\Sigma^{V\oplus V'}(Z_+)\ar[r,"\Delta"]&\Sigma^{V\oplus V'}((Z\times Z)_+)\ar[r,"\cong"]&\Sigma^V(Z_+)\wedge \Sigma^{V'}(Z_+)\\
{}\ar[r,"J(\xi)\wedge J(\xi')"]&S^W\wedge S^{W'}\ar[r,"\cong"]&S^{W\oplus W'}
\end{tikzcd},\end{center}
which exactly corresponds to $J(\xi)\cdot J(\xi')$. 
\end{proof}

\begin{ex}
There is a cofiber sequence $Z_+\rightarrow S^0\rightarrow SZ$, and precomposing $J$ with restriction along the boundary map $SZ\rightarrow\Sigma(Z_+)$ gives an equivariant $J$-homomorphism of signature
\[
KF_G^{-1}(Z) \rightarrow \pi_0 D(\Sigma^\infty_+ Z)^\times,
\]
as considered, for example, in \cite{segal1971equivariant}.
\tqed
\end{ex}

\begin{ex}
Write
\[
\tilde{J} = \partial\circ J \colon \widetilde{KF}{}_G^0(SZ)\rightarrow \pi_{\star-1}^GD(\Sigma^\infty SZ).
\]
If $\lambda$ is a $G$-representation and $Z = S^{\lambda+1}$, then $a_Z = 0$ and $\tilde{J}$ recovers the equivariant $J$-homomorphism of signature
\[
\pi_{\lambda+2}KF_G\rightarrow\pi_{\lambda+1} S_G,
\]
as considered, for example, in \cite{loffler1977equivariant,crabb1980z2,minami1983equivariant}.
\tqed
\end{ex}

\begin{rmk}\label{rmk:tildej}
In general, if $Z$ is $G$-pointed then $SZ\simeq \Sigma Z$ and $a_Z = 0$, so the $J$-homomorphism is of the form
\[
\widetilde{KF}{}_G^{-1}(Z)\rightarrow \pi_0^G D(\Sigma^\infty_+ Z).
\]
The basepoint of $Z$ induces a splitting $D(\Sigma^\infty_+ Z)\simeq S_G\oplus D(\Sigma^\infty Z)$, and so we may project onto the latter summand to obtain
\[
\tilde{J} = \partial\circ J\colon \widetilde{KF}{}_G^{-1}(Z)\rightarrow \pi_0^G D(\Sigma^\infty Z).
\]
Note that this need not be a group homomorphism, though it is if $Z$ is $G$-connected. On the other hand, one can also consider the (perhaps more familiar) $J$-homomorphism
\[
J'\colon \widetilde{KF}{}_G^{-1}(Z)\rightarrow \pi_0^G D(\Sigma^\infty Z)
\]
sending a \textit{pointed} map $\phi\colon Z\rightarrow L(V,V)$ to the adjoint $\Sigma^VZ\rightarrow S^V$ of the composite $j\circ\phi\colon Z\rightarrow\Map_\ast(S^V,S^V)$. If $Z$ is $G$-connected then $J' = \tilde{J}$, but they can differ in general, as the next example demonstrates.
\tqed
\end{rmk}

\begin{ex}
Take $Z = S^0$. The homomorphism $j\colon RO(G)\rightarrow A(G)^\times$ indicated in \cref{rmk:signs} factors through the surjection $\eta\colon RO(G) = \pi_0 KO_G\rightarrow \pi_1 KO_G$, and we have
\[
J'(\eta\cdot \alpha) = j(\alpha),\qquad \tilde{J}(\eta\cdot\alpha) = \pm(1-j(\alpha)).
\]
For example $\tilde{J}(\eta) = \pm 2$, and if $\sigma$ is a $1$-dimensional representation with index $2$ kernel $K\subset G$ then $\tilde{J}(\eta\cdot\sigma) = \pm \tr_K^G(1)$.
\tqed
\end{ex}

\section{The equivariant Atiyah--Hirzebruch spectral sequence}\label{sec:ahss}

If $Z$ is a $G$-complex and $R$ is a $G$-spectrum, then there is an \textit{Atiyah--Hirzebruch spectral sequence}
\[
E_2^{\alpha,f}(Z;R) = H^f_G(Z;\ul{\pi}_{\alpha+f} R) \Rightarrow R^{-\alpha}_G(Z),
\]
with $E_2$-page the Bredon cohomology of $Z$ with coefficients in the Mackey functor $\ul{\pi}_\star R$.

\begin{ex}
For $Z = EG$ one obtains the \textit{homotopy fixed point spectral sequence}
\[
E_2^{\alpha,f} = H^f(G;\pi_{\alpha+f}^e R) \Rightarrow  \pi_\alpha^G F(EG_+,R)\cong \pi_{0} \left(F(S^{\alpha},R)^{\h G}\right),
\]
with $E_2$-page the group cohomology of $G$ with coefficients in the underlying homotopy groups $\pi_{\star}^e R$.
\tqed
\end{ex}

In this section we apply the $J$-homomorphism constructed in \cref{sec:jhomomorphism} to obtain universal periodicities and differentials in these spectral sequences.

\subsection{A technical lemma}

We will need a technical lemma. Let $p\colon Z\rightarrow T$ be a map of $G$-spaces with a homotopy retraction $i\colon T\rightarrow Z$, so that $p\circ i\simeq \text{id}_T$. In particular, there is an equivalence
\[
T/Z\simeq S(Z/T).
\]
Write
\[
q\colon Z\rightarrow Z/T,\qquad \partial\colon S(Z/T)\simeq T/Z\rightarrow SZ
\]
for the canonical maps, and note that
\[
Sq\circ \partial\colon S(Z/T)\rightarrow SZ\rightarrow S(Z/T)
\]
is an equivalence. The following diagram of coCartesian squares may help illustrate the situation:
\begin{center}\begin{tikzcd}
T\ar[r]\ar[d]&Z\ar[r]\ar[d]&T\ar[r]\ar[d]&\ast\ar[d]\\
\ast\ar[r]&Z/T\ar[r]\ar[d]&\ast\ar[r]\ar[d]&ST\ar[r]\ar[d]&\ast\ar[d]\\
&\ast\ar[r]&S(Z/T)\simeq T/Z\ar[r]&SZ\ar[r]&SZ/ST\simeq S(Z/T)
\end{tikzcd}.\end{center}

The retraction $p\circ i \simeq \text{id}_T$ induces a stable splitting
\[
D(\Sigma^\infty_+ Z)\simeq D(\Sigma^\infty_+ T)\oplus D(\Sigma^\infty Z/T).
\]
Here, $D(\Sigma^\infty Z/T)$ is a module over $D(\Sigma^\infty_+ Z)$, and thus over $D(\Sigma^\infty_+ T)$ by restriction along $p$. We now have the following.

\begin{lemma}\label{lem:jr}
If $\xi \in \widetilde{KO}{}_G^0(SZ)$, then under the above splitting we have
\[
J(\xi) = (J(Si^\ast\xi),~\tilde{J}(\partial^\ast(\xi))\cdot J(Si^\ast \xi)),
\]
where $\tilde{J}$ is as in \cref{rmk:tildej}.
\end{lemma}
\begin{proof}
The relevant diagram is
\begin{center}\begin{tikzcd}
\widetilde{KO}{}_G^0(S(Z/T))\ar[r,"J"]\ar[d,"Sq^\ast",shift left=1mm]&\pi_0 D(\Sigma^\infty_+ Z/T)\ar[r,"\simeq"]\ar[d,"q^\ast",shift left=1mm]&\pi_0 \left(S_G\oplus D(\Sigma^\infty Z/T)\right)\ar[d,"\iota\oplus \text{id}"]\\

\widetilde{KO}{}_G^0(SZ)\ar[d,"Si^\ast",shift left=1mm]\ar[r,"J"]\ar[u,"\partial^\ast",dashed,shift left=1mm]&\pi_\star  D(\Sigma^\infty_+ Z)\ar[r,"\simeq"]\ar[d,"i^\ast",shift left=1mm]&\pi_\star\left( D(\Sigma^\infty_+ T)\oplus D(\Sigma^\infty Z/T)\right)\ar[dl,"\text{pr}_1",shift left=1mm]\\

\widetilde{KO}{}_G^0(ST)\ar[r,"J"]\ar[u,"Sp^\ast",shift left=1mm]&\pi_\star D(\Sigma^\infty_+ T)\ar[u,"p^\ast",shift left=1mm]\ar[ur,"\text{in}_1",shift left=1mm]
\end{tikzcd}.\end{center}
First suppose $\xi = Sp^\ast\zeta$ for some $\zeta \in \widetilde{KO}{}_G^0(ST)$. In this case we have
\[
J(\xi) = J( Sp^\ast\zeta) = p^\ast J(\zeta) = (J(\zeta),0) = (J(Si^\ast\xi),0)
\]
as claimed. Next suppose $Si^\ast \xi = 0$. In this case $\xi = Sq^\ast \partial^\ast \xi$ and thus
\[
J(\xi) = J(Sq^\ast\partial^\ast\xi) = q^\ast J(\partial^\ast\xi).
\]
Under the splitting
\[
D(\Sigma^\infty_+ Z/T)\simeq S_G\oplus D(\Sigma^\infty Z/T)
\]
we have by definition
\[
J(\partial^\ast\xi) = (1,\tilde{J}(\partial^\ast\xi)),
\]
and thus
\[
q^\ast J(\partial^\ast\xi) = q^\ast (1,\tilde{J}(\partial^\ast\xi)) = (1,\tilde{J}(\partial^\ast\xi)).
\]
Finally, for general $\xi$ we combine the preceding cases to compute
\begin{align*}
J(\xi) &= J(\xi - Sp^\ast Si^\ast\xi) \cdot J(Sp^\ast Si^\ast \xi) \\
&= (1,\tilde{J}(\partial^\ast(\xi - Sp^\ast Si^\ast\xi))) \cdot (J(Si^\ast Sp^\ast Si^\ast \xi),0) \\
&= (1,\tilde{J}(\partial^\ast \xi))\cdot (J(Si^\ast \xi),0) = (J(Si^\ast\xi), \tilde{J}(\partial^\ast \xi)\cdot J(Si^\ast\xi))
\end{align*}
as claimed. 
\end{proof}

\subsection{Periodicities and differentials in the equivariant AHSS}
A \textit{$G$-complex} is a filtered $G$-space
\[
Z = \colim_{n\rightarrow\infty}Z^{\leq n},
\]
where $Z^{\leq n}$ is obtained from $Z^{<n}$ by attaching $n$-cells. For our purposes, this will mean that we have specified $G$-sets $I_n$ for each $n\geq 0$, together with pushout squares
\begin{center}\begin{tikzcd}
S^{n-1}\times I_n\ar[r]\ar[d]&Z^{<n}\ar[d]\\
D^n\times I_n\ar[r]&Z^{\leq n}
\end{tikzcd}\end{center}
for $n\geq 0$. Note in particular
\[
Z^{\leq n}/Z^{<n}\simeq (D^n\times I_n)/(S^{n-1}\times I_n)\simeq \Sigma^n(I_{n+}).
\]
For a $G$-set $I$ and $G$-spectrum $R$, write
\[
(\ul{\pi}_\star R)(I) = \pi_\star^G F(\Sigma^\infty_+ I,R)
\]
for the evaluation of the Mackey functor $\ul{\pi}_\star R$ on $I$. Then the \textit{Atiyah--Hirzebruch spectral sequence}
\[
E_1^{\alpha,f}(Z;R) = (\ul{\pi}_{\alpha+f}R)(I_f)\Rightarrow R^{-\alpha}_G(Z)
\]
is the spectral sequence associated to the filtration 
\[
F(\Sigma^\infty_+Z,R)\simeq \lim_{f\rightarrow\infty}F(\Sigma^\infty_+ Z^{<f},E).
\]
The $E_2$-page is given by the Bredon cohomology groups
\[
E_2(Z;R) = H^\ast_G(Z;\ul{\pi}_\star R).
\]
If $f\geq 1$ then the $J$-homomorphism defines a map
\[
\tilde{J}\colon E_1^{0,f}(Z;KO_G) = (\ul{\pi}_f KO_G)(I_f)\rightarrow (\ul{\pi}_{f-1} S_G)(I_f) = E_1^{-1,f}(Z;S_G).
\]
In general $E_r(Z;R)$ is a module over $E_r(Z;S_G)$, and we have the following.

\begin{theorem}\label{thm:ahss}
Fix $\alpha \in RO(G)$, and suppose that the image of $\alpha$ in $KO_G^0(Z)$ is detected by $b\in E_1^{0,f}(Z;KO_G)$ with $f\geq 1$. Then there exists an invertible element $t_\alpha \in E_1^{\alpha,0}(Z;R)$ which survives to the $E_f$-page, whereupon
\[
d_f(t_\alpha) = \pm \tilde{J}(b) t_\alpha
\]
in $E_1^{\alpha-1,f}(Z;R)$, where $\pm$ refers to signs in the sense of \cref{rmk:signs}.
\end{theorem}
\begin{proof}
We may as well suppose $R = S_G$. The relevant diagram is the following:
\begin{center}\begin{tikzcd}[column sep=1mm, row sep=tiny]
\widetilde{KO}{}_G^0(S(S^{f-1}\wedge I_{f+}))\ar[rr,equals]\ar[dr,"J"]\ar[dd, shift left=1mm]&&(\ul{\pi}_f KO_G)(I_f)\ar[r,"\tilde{J}"]\ar[from=dd, "\partial", dashed]&(\ul{\pi}_{f-1}S_G)(I_f)\\

&\pi_\star D((S^{f-1}\wedge I_{f+})_+)
&&\pi_\star D(\Sigma^{f-1}I_{f+})\\

\widetilde{KO}{}_G^0(S(S^{f-1}\times I_f))\ar[dd, shift left=1mm, "Si^\ast"]\ar[dr,"J"]\ar[uu,"\partial^\ast",shift left=1mm, dashed]&&\widetilde{KO}{}_G^0(SZ^{<f})\ar[dr,"J"]\ar[from=dd]\ar[ll, "Se^\ast"', near start]\\

&\pi_\star D((S^{f-1}\times I_f)_+)\ar[from=uu, crossing over, shift left=1mm]\ar[dd,shift left=1mm]
&&\pi_\star D(Z^{<f}_+)\ar[ll, crossing over,"e^\ast"', near end]\ar[uu,dashed,"\partial"]\\

\widetilde{KO}{}_G^0(SI_f)\ar[dr,"J"]\ar[uu,shift left=1mm]&&\widetilde{KO}{}_G^0(SZ^{\leq f})\ar[dr,"J"]\ar[ll]\\

&\pi_\star D(I_{f+})&&\pi_\star D(Z^{\leq f}_+)\ar[uu]\ar[ll]
\end{tikzcd}\end{center}

Here we have abbreviated $D(X) = D(\Sigma^\infty X)$ for a pointed $G$-space $X$. 

As $\alpha$ is detected by $b\in E_1^{0,f}(Z;KO_G)$, there exists $\tilde{b} \in \widetilde{KO}{}^0_G(SZ^{<f})$ satisfying $\partial(\tilde{b}) = b$ and restricting to $\alpha$ along $S^0\rightarrow SZ^{<f}$. Applying the $J$-homomorphism we obtain an invertible element
\[
J(\tilde{b}) \in \pi_\alpha^G D(\Sigma^\infty_+ Z^{< f}).
\]
We take $t_\alpha$ to be the image of $J(\tilde{b})$ under restriction along $I_0 = Z^{\leq 0}\rightarrow Z^{< f}$; we will also abuse notation by writing the same for the image of $J(\tilde{b})$ under restriction along the map $i_f\colon I_f\rightarrow Z^{< f}$. Then to say that $d_f(t_\alpha) = \tilde{J}(b)\cdot t_\alpha$ is to say that $\partial(J(\tilde{b})) = \tilde{J}(b)\cdot t_\alpha$.

Under the splitting
\[
D((S^{f-1}\times I_{f+})_+)\simeq D(I_{f+}) \oplus D(\Sigma^{f-1}I_{f+})
\]
we may identify
\[
e^\ast(x) = (i_f^\ast(x),\partial(x))
\]
for any $x\in \pi_\star D(Z^{<f}_+)$. Using \cref{lem:jr}, we compute
\begin{align*}
e^\ast J(\tilde{b}) = J(Se^\ast\tilde{b}) 
&= (J(Si^\ast (Se)^\ast \tilde{b}),~ \tilde{J}(\partial^\ast Se^\ast \tilde{b})\cdot J(Si^\ast Se^\ast \tilde{b})).
\end{align*}
Here, by construction we have $\partial^\ast (Se)^\ast \tilde{b} = b$ and $J(Si^\ast (Se)^\ast\tilde{b}) = t_\alpha$, implying that
\[
e^\ast J(\tilde{b}) = (t_\alpha,\tilde{J}(b)\cdot t_\alpha),
\]
and so $\partial(J(\tilde{b})) = \tilde{J}(b)t_\alpha$ as claimed.
\end{proof}

\section{Equivariant \texorpdfstring{$t_\alpha$}{t\_alpha}-self maps and \texorpdfstring{$t_\alpha$}{t\_alpha}-elements}\label{sec:bordism}

This section contains our study of $t_\alpha$-elements and $t_\alpha$-self maps, connecting the periodicities produced by the $J$-homomorphism of \cref{sec:jhomomorphism} to periodicities in equivariant cobordism.

\subsection{Equivariant cobordism and \texorpdfstring{$t_\alpha$}{t\_alpha}-elements}\label{ssec:cobordism}

We begin by recalling the definitions needed to make sense of $t_\alpha$-elements. Let $F$ denote one of the real division algebras $\bbR$, $\bbC$, or $\bbH$, and write
\[
MF_G = \begin{cases}
MO_G,&F=\bbR,\\
MU_G,&F=\bbC,\\
MSp_G,&F=\bbH
\end{cases}
\]
for the corresponding $G$-equivariant homotopical cobordism spectrum. This was originally constructed by tom Dieck \cite{dieck1970bordism} when $F = \bbC$, and the construction works just as well for $F = \bbR$ and $F = \bbH$. To be precise tom Dieck constructs $MF_G$ as a $G$-equivariant cohomology theory; a good construction of $MF_G$ as a $G$-spectrum suitable for our purposes may be found in \cite{sinha2001computations}. The $G$-spectra $MF_G$ are compatible as $G$ varies, and assemble together to form a highly structured globally equivariant ring spectrum. We refer the reader to \cite[Chapter 6]{schwede2018global} for a careful construction of $MF_G$ which incorporates this additional structure.

A vector bundle $\gamma$ over a compact $G$-space $Z$ has associated sphere bundle $S(\gamma)$ and disk bundle $D(\gamma)$, and the \textit{Thom space} of $\gamma$ is given by
\[
\Th(\gamma) = D(\gamma)/S(\gamma).
\]
This is a pointed $G$-space, and we will generally not distinguish between it and its suspension spectrum. Associated to $\gamma$ is a \textit{Thom class}
\[
t_\gamma \in \widetilde{MF}{}_G^{-|\gamma|}\Th(\gamma),
\]
cupping with which induces the \textit{Thom isomorphism}
\[
\widetilde{MF}{}_G^\star\Th(\gamma)\cong MF_G^{\star+|\gamma|}(Z).
\]
Moreover, $MF_G$ is the universal $G$-spectrum with such Thom isomorphisms \cite{okonek1982connerfloyd}. 

We only need the simplest example of a Thom class. Given a $G$-representation $V$, we write 
\[
t_V \in \pi_{V-|V|}MF_G
\]
for the Thom class of $V$ considered as a vector bundle over a point, noting that $\Th(V) = S^V$. This satisfies
\[
t_{V\oplus W} = t_V t_W,\qquad t_{F^n} = 1,
\]
and has inverse the orientation class $u_V\in \pi_{|V|-V}MF_G$. Moreover $t_V$ depends only on the isomorphism class of $V$, allowing for the following definition.

\begin{defn}
The \textit{Thom class} $t_\alpha \in \pi_{\alpha-|\alpha|}MF_G$ of a virtual $G$-representation $\alpha = V-W$ is defined as $t_\alpha = t_V t_W^{-1}$.
\tqed
\end{defn}

From now on, we shall assume that $\alpha$ has virtual dimension zero. This loses no real generality for our purposes, and is convenient as it ensures $|t_\alpha| = \alpha$. 

\begin{defn}\label{def:talpha}
A \textit{$t_\alpha$-element} of order $n>0$ in a $G$-ring spectrum $R$ is an invertible element $t \in \pi_{n\alpha}^G R$ lifting $t_\alpha^n$ under the Hurewicz map $R\rightarrow R \otimes MU_G$.
\tqed
\end{defn}

The condition that a $t_\alpha$-element is invertible is automatic if, for example, $R$ is $MF_G$-local. We include it as our examples satisfy it.

\begin{defn}
A \textit{$t_\alpha$-self map} of order $n\geq 1$ on a $G$-spectrum $X$ is an equivalence
\[
f\colon \Sigma^{n\alpha}X\rightsim X
\]
which induces multiplication by $t_\alpha^n$ after smashing with $MF_G$, that is for which $MF_G\otimes f = t_\alpha^n \otimes X$ as self maps of $MF_G\otimes X$.
\tqed
\end{defn}


We record the following for easy reference.

\begin{lemma}\label{lem:closure}
If $t\in \pi_{n\alpha}R$ is a $t_\alpha$-element and $M$ is an $R$-module, then multiplication by $t$ defines a $t_\alpha$-self map $\Sigma^{n\alpha}M\rightarrow M$. Conversely, every $R$-linear $t_\alpha$-self map $\Sigma^{n\alpha}R\rightarrow R$ is given by multiplication with a $t_\alpha$-element.
\end{lemma}
\begin{proof}
Immediate from the definitions.
\end{proof}

\subsection{Vector bundles and \texorpdfstring{$K$}{K}-theory}\label{ssec:k}

Let $Z$ be a compact $G$-space. We now relate the study of $t_\alpha$-elements in $D(\Sigma^\infty_+Z)$ to the study of vector bundles on $Z$.

Given a $G$-representation $V$, write $Z_V$ for the vector bundle $V\times Z \rightarrow Z$. This has Thom spectrum
\[
\Th(Z_V)\simeq \Sigma^V\Sigma^\infty_+Z.
\]
In particular, a stable equivalence $Z_V\cong Z_W$ of vector bundles over $Z$ induces an equivalence $\Sigma^V\Sigma^\infty_+ Z \simeq \Sigma^W\Sigma^\infty_+ Z$. The basic observation is that this is determined by a $t_{V-W}$-element in $D(\Sigma^\infty_+ Z)$, and that such elements are parametrized by the $J$-homomorphism of \cref{sec:jhomomorphism}, as we now explain.

If $\xi$ is a vector bundle over $Z$, then the Thom diagonal $ \Th(\xi)\rightarrow \Th(\xi)\otimes \Sigma^\infty_+ Z$ transposes to make $\Th(\xi)$ into a module over the Spanier--Whitehead dual $D(\Sigma^\infty_+ Z)$. When $\xi = Z_V$, one obtains the usual module structure on $\Th(Z_V) = \Sigma^V\Sigma^\infty_+Z$. In particular, a map $Z_V\rightarrow Z_W$ of vector bundles induces a $D(\Sigma^\infty_+Z)$-linear map $\Sigma^V\Sigma^\infty_+Z\rightarrow\Sigma^W\Sigma^\infty_+Z$. By duality, such a map is given by capping with an element of $\pi_{V-W}^GD(\Sigma^\infty_+Z)$.

\begin{rmk}
Explicitly, given a map $f\colon \Sigma^V\Sigma^\infty_+Z\rightarrow \Sigma^W\Sigma^\infty_+Z$ one defines $t\in\pi_{V-W}^GD(\Sigma^\infty_+Z)$ to be the composite
\begin{center}\begin{tikzcd}
\Sigma^V\Sigma^\infty_+Z\ar[r,"f"]&\Sigma^W\Sigma^\infty_+Z\ar[r]&\Sigma^W\Sigma^\infty_+(\ast)\ar[r,"\simeq"]&S^W
\end{tikzcd},\end{center}
and when $f$ is $D(\Sigma^\infty_+Z)$-linear it may be recovered from $t$ as the composite
\begin{equation*}\begin{tikzcd}
\Sigma^V\Sigma^\infty_+Z\ar[r,"\Sigma^V\Delta"]&\Sigma^V\Sigma^\infty_+Z\otimes\Sigma^\infty_+Z \ar[r,"t\otimes \Sigma^\infty_+Z"]&\Sigma^W\Sigma^\infty_+Z
\end{tikzcd}.\eqno\triangleleft\end{equation*}
\end{rmk}

We can now give the following.

\begin{lemma}\label{lem:fiberwise}
Fix an equivalence $Z_V\rightsim Z_W$ of vector bundles over $Z$ with associated equivalence $f\colon \Sigma^V\Sigma^\infty_+ Z \rightsim\Sigma^W\Sigma^\infty_+Z$ of Thom spectra. Then $f$ is a $t_{V-W}$-self map and the associated class $t\in \pi_{V-W}^GD(\Sigma^\infty_+Z)$ is a $t_{V-W}$-element.
\end{lemma}
\begin{proof}
By the above discussion and \cref{lem:closure}, it suffices to show that $t$ is a $t_{V-W}$-element. By definition $t = f^\ast(1)$, where $f^\ast\colon MF_G^\ast(\Sigma^W\Sigma^\infty_+ Z)\rightarrow MF_G^\ast(\Sigma^V\Sigma^\infty_+ Z)$. By $MF_G^\star$-linearity of $f^\ast$ and naturality of Thom classes, we can identify
\[
t = f^\ast(1) = f^\ast(t_W^{-1}t_W) = t_W^{-1}f^\ast(t_W) = t_W^{-1}t_{f^\ast W} = t_W^{-1}t_V = t_{V-W}
\]
in $MF_G^{W-V}(Z)$, and so $t$ is a $t_{V-W}$-element as claimed.
\end{proof}

\begin{prop}\label{prop:jt2}
Consider the $J$-homomorphism
\[
J\colon \widetilde{KF}{}_G^0(SZ)\rightarrow \pi_\star^G D(\Sigma^\infty_+ Z)^\times.
\]
If $\xi \in \widetilde{KF}{}_G^0(SZ)$ satisfies $a_Z \xi = \alpha \in RF(G)$, then $J(\xi) \in \pi_\alpha^G D(\Sigma^\infty_+ Z)$ is a $t_\alpha$-element.
\end{prop}
\begin{proof}
Write $\alpha = V-W$ for two $G$-representations $V$ and $W$. By construction, $J(\xi)$ is obtained by Thomifying a stable equivalence $Z_{V}\simeq Z_{W}$, i.e.\ an equivalence of vector bundles $Z_{V\oplus U}\simeq Z_{W\oplus U}$ for some $G$-represenation $U$. It follows from \cref{lem:fiberwise} that $J(\xi)$ is a $t_{(V\oplus U)-(W\oplus U)} = t_\alpha$-element as claimed.
\end{proof}

\begin{cor}\label{cor:existk}
There exists a $t_\alpha$-element in $\pi_\alpha D(\Sigma^\infty_+ Z)$ provided $\alpha$ is in the kernel of $RF(G)\rightarrow KF_G^0(Z)$.
\end{cor}
\begin{proof}
By the cofiber sequence $Z_+\rightarrow S^0\rightarrow SZ$, if $\alpha\in\ker(RF(G)\rightarrow KF_G^0(Z))$ then $\alpha = a_Z b$ for some $b\in \widetilde{KF}{}_G^0(SZ)$, and by \cref{prop:jt2} $J(b) \in \pi_\alpha^G D(\Sigma^\infty_+ Z)$ is our desired $t_\alpha$-element.
\end{proof}

\subsection{Character theory}

Suppose that $G$ is finite. In this case, we can use \cref{cor:existk} to give criteria for $D(\Sigma^\infty_+ Z)$ to admit a $t_\alpha$-element of some order based only on the isotropy groups of $Z$. We will extend this to arbitrary $G$-ring spectra in \cref{ssec:maint} below.

We require some general equivariant localization theory.

\begin{lemma}\label{lem:macksplit}
Write $\Mack_G$ for the category of $G$-Mackey functors, and for $H\subset G$ write $W_GH = N_GH/H$ for the Weyl group. Then there are functors
\[
\tau_H\colon \Mack_G\rightarrow\Mod_{\bbZ[W_GH]},\quad \tau_H \ul{M} = \coker\left(\tr\colon \bigoplus_{K\subsetneq H}\ul{M}(G/K)\rightarrow \ul{M}(G/H)\right)
\]
for which the following hold:
\begin{enumerate}
\item For any $G$-Mackey functor $\ul{M}$, there is a natural splitting
\[
\bbZ[\tfrac{1}{|G|}]\otimes \ul{M}(G)\cong \bbZ[\tfrac{1}{|G|}]\otimes \bigoplus_{(H)}(\tau_H \ul{M})^{W_GH},
\]
this sum being over the conjugacy classes of subgroups of $G$.
\item If $E$ is a $G$-spectrum, then the geometric fixed point maps $\phi_H\colon \pi_0^G E\rightarrow \pi_0 \Phi^H E$ factor through $\tau_H (\ul{\pi}_0 E)$, and induce isomorphisms
\[
\bbZ[\tfrac{1}{|G|}]\otimes \tau_H (\ul{\pi}_0 E)\cong \bbZ[\tfrac{1}{|G|}]\otimes \pi_0 \Phi^H E.
\]
\end{enumerate}
\end{lemma}
\begin{proof}
See for example \cite[Section 3.4]{schwede2018global}.
\end{proof}

\begin{lemma}\label{lem:gtors}
Let $E$ be a $G$-spectrum. Fix $u\in E^0_G(Z)$, and for $H\subset G$ abbreviate $u_H = \res^G_H u$. Suppose that
\[
Z^H\neq \emptyset \implies u_H = 0
\]
for all $H\subset G$. Then $u$ has finite order dividing a power of $|G|$.
\end{lemma}
\begin{proof}
As $Z$ is compact we have $\Phi^H \Sp^G(\Sigma^\infty_+Z,E) \cong \Sp^G(\Sigma^\infty_+(Z^H),\Phi^H E)$, and so \cref{lem:macksplit} implies
\[
E^0_G(Z)[\tfrac{1}{|G|}]\subset \prod_{H\subset G}(\Phi^H E)^0(Z^H)[\tfrac{1}{|G|}].
\]
It thus suffices to show that $u_H=0$ in $(\Phi^H E)^0(Z^H)$ for all $H\subset G$. If $Z^H = \emptyset$ then this group vanishes. Otherwise $u_H = 0$ by assumption. In either case this shows that $u=0$ in $E^0_G(Z)[\tfrac{1}{|G|}]$, and so $u$ has finite order dividing a power of $|G|$.
\end{proof}

Combining these two lemmas leads to the following.

\begin{prop}\label{prop:mainlem}
A class $\alpha \in KF_G^0(Z)$ has finite order if and only if $Z^C\neq\emptyset\Rightarrow\alpha_C=0$ for all cyclic $C\subset G$, in which case its order divides a power of $|G|$.
\end{prop}
\begin{proof}
First suppose that $Z^C\neq\emptyset \Rightarrow \alpha_C = 0$ for all cyclic $C\subset G$. If $Z^H \neq \emptyset$, then $Z^C\neq\emptyset$ for all cyclic $C\subset H$. It follows that $\alpha_C = 0$ for all cyclic $C\subset H$, and as a representation is determined by its restriction to cyclic subgroups we deduce $\alpha_H = 0$. Thus $Z$ satisfies $Z^H \neq \emptyset \Rightarrow \alpha_H=0$ for all $H\subset G$. By \cref{lem:gtors}, this implies that $\alpha \in KF_G^0(Z)$ has finite order dividing a power of $|G|$.

Conversely, if $Z^C \neq \emptyset$, then there is some equivariant map $p\colon G/C\rightarrow Z$. This must satisfy $p^*(\alpha) = \alpha_C \in KF_G^0(G/C) \cong RF(C)$. It follows that if $\alpha$ has finite order, then so does $\alpha_C \in RF(C)$, implying that $\alpha_C = 0$ as $RF(C)$ is torsion-free.
\end{proof}

\begin{cor}\label{cor:mainlem}
If $Z^C \neq\emptyset\Rightarrow\alpha_C = 0$ for all cyclic subgroups $C\subset G$, then $D(\Sigma^\infty_+Z)$ admits a $t_\alpha$-element of order dividing a power of $|G|$.
\end{cor}
\begin{proof}
Combine \cref{prop:mainlem} and \cref{cor:existk}.
\end{proof}

\subsection{\texorpdfstring{$\calF$}{F}-nilpotence}\label{ssec:nilpotence}

Our next goal is to upgrade \cref{cor:mainlem} to arbitrary $G$-ring spectra. To do this we will make use of the nilpotence machinery developed by Mathew--Nauman--Noel in \cite{mathewnaumannnoel2017nilpotence,mathewnaumannnoel2019derived}. In this subsection we extract the parts of this theory that we will need.

Given a subgroup $H\subset G$, abbreviate
\[
G/H_+ = \Sigma^\infty_+ G/H.
\]
Let $\calF$ be a family of subgroups of $G$, i.e.\ a collection of subgroups closed under subconjugacy. 

\begin{defn}[{\cite[Definition 6.36]{mathewnaumannnoel2017nilpotence}}]
A $G$-spectrum $X$ is \textit{$\calF$-nilpotent} if it lies in the thick $\otimes$-ideal generated by $G/H_+$ for $H\in\calF$.
\tqed
\end{defn}

Up to $G$-homotopy equivalence, there is a unique $G$-space $E\calF$ characterized by
\[
E\calF^H \simeq \begin{cases}\emptyset,&H\notin\calF,\\\ast,&H\in\calF.\end{cases}
\]
A convenient model for this space is given as follows. Write
\[
G/\calF = \coprod_{H\in\calF}G/H.
\]
Then $E\calF$ is equivalent to an infinite join of copies of $G/\calF$:
\[
E\calF = \colim\left(G/\calF\rightarrow G/\calF \ast G/\calF \rightarrow G/\calF\ast G/\calF \ast G/\calF\rightarrow\cdots\right).
\]
Write $E\calF^{<m} = (G/\calF)^{\ast m-1}$ for the resulting $(m-1)$-skeleton of $E\calF$. This satisfies
\[
(E\calF^{<m})^H \simeq ( (G/\calF)^H)^{\ast m-1} \simeq \begin{cases}
\emptyset, & H \notin\calF ,\\
\bigvee S^{m-1},& H \in \calF.
\end{cases}
\]
We also abbreviate $E\calF^{<m}_+ = \Sigma^\infty_+ E\calF^{<m}$.

\begin{lemma}\label{lem:nildef}
Let $X$ be a $G$-spectrum, and consider the following statements.
\begin{enumerate}
\item $X$ is $\calF$-nilpotent.
\item $E\calF_+^{< m} \otimes X \rightarrow X$ admits a section for some $m$.
\item $X\rightarrow \Sp^G(E\calF_+^{<m},X) \simeq D(E\calF_+^{<m}) \otimes X$ admits a retraction for some $m$.
\item $\Phi^HX \neq 0 \Rightarrow H \in\calF$.
\end{enumerate}
Always (1)$\Leftrightarrow$(2)$\Leftrightarrow$(3)$\Rightarrow$(4). If $X$ is finite or a $G$-ring spectrum, then (4)$\Rightarrow$(2).
\end{lemma}
\begin{proof}
(1)$\Leftrightarrow$(2)$\Leftrightarrow$(3): See \cite[Theorem 2.25, Remark 2.27]{mathewnaumannnoel2019derived}.

(2)$\Rightarrow$(4): If $H\notin\calF$ then 
\[
\Phi^H(E\calF_+^{<m}\otimes X)\simeq \Phi^H(E\calF_+^{<m})\otimes\Phi^H X \simeq 0 \otimes \Phi^H X \simeq 0.
\]
Thus if $X$ is a retract of $E\calF_+^{<m}\otimes X$ then $\Phi^H X \simeq 0$ for $H\notin\calF$.

(4)$\Rightarrow$(2) when $X$ is finite: If $\Phi^H X \neq 0 \Rightarrow H \in \calF$, then the map $E\calF_+ \otimes X \rightarrow X$ is an equivalence. As $X$ is finite, the inverse $X\rightarrow E\calF_+ \otimes X \simeq \colim_{n\rightarrow\infty}E\calF_+^{< m}\otimes X$ factors through some $E\calF_+^{< m}\otimes X$.

(4)$\Rightarrow$(2) when $X$ is a ring: See \cite[Theorem 6.41]{mathewnaumannnoel2017nilpotence}.
\end{proof}

This allows for the following explicit quantification of $\calF$-nilpotence, see \cite[Definition 6.36]{mathewnaumannnoel2017nilpotence} and \cite[Proposition 2.26]{mathewnaumannnoel2019derived}

\begin{defn}
The \textit{$\calF$-exponent} $\exp_\calF(X)$ of an $\calF$-nilpotent $G$-spectrum $X$ is the minimal $m$ for which $X\rightarrow D(E\calF_+^{<m})\otimes X$ admits a retraction.
\tqed
\end{defn}

If $R$ is a $\calF$-nilpotent $G$-ring spectrum with $\exp_\calF(R)\leq m$, then precomposing the guaranteed retraction $D(E\calF_+^{<m})\otimes R\rightarrow R$ with the unit $D(E\calF_+^{<m})\rightarrow D(E\calF_+^{<m})\otimes R$ provides a map
\[
r\colon D(E\calF_+^{<m})\rightarrow R
\]
satisfying $r(1) = 1$. This map is \textit{not} guaranteed to be multiplicative. It is possible to show by a formal argument that the composite $D(E\calF_+^{<2m})\rightarrow D(E\calF_+^{<m})\rightarrow R$ is multiplicative, and that although the composite $D(E\calF_+^{<m+1})\rightarrow D(E\calF_+^{<m})\rightarrow R$ may fail to be multiplicative the induced map on homotopy groups preserves invertible elements. However, in our particular context it turns out we can do just a little better.

\begin{lemma}\label{prop:nishida}
Let $R$ be a $G$-ring spectrum. Fix a map $r\colon D(E\calF_+^{<m})\rightarrow R$ satisfying $r(1) = 1$, and fix $t\in \pi_\alpha ^GD(E\calF_+^{<m})$.
\begin{enumerate}
\item If $t$ lifts $t_\alpha \in \pi_\alpha^G(MF_G\otimes D(E\calF_+^{<m}))$, then $r(t)$ lifts $t_\alpha \in \pi_\alpha^G(MF_G \otimes R)$.
\item If $t$ is invertible and $m\geq 2$, then $r(t) \in \pi_\alpha^G R$ is invertible.
\end{enumerate}
In particular, if $m\geq 2$ then $r$ preserves $t_\alpha$-elements.
\end{lemma}

\begin{proof}
(1)~~This holds as the maps $D(E\calF_+^{<m})\rightarrow MF_G \otimes D(E\calF_+^{<m})$ and $R\rightarrow MF_G\otimes R$ are defined using only the unit maps of $D(E\calF_+^{<m})$ and $R$, and $r$ is compatible with these.

(2)~~It suffices to show that $\Phi^H r(t)$ is invertible for all $H\subset G$. Abbreviate $W = (E\calF^{<m})^H$. As $t\in \pi_\alpha ^GD(E\calF_+^{<m})$ is invertible, necessarily $\Phi^H t \in \pi_{\alpha^H}D(\Sigma^\infty_+W)$ is invertible, implying also $|\alpha^H| = 0$. It now suffices to prove the nonequivariant assertion that if $T = \Phi^H R$ is a ring spectrum, $f = \Phi^H r\colon D(\Sigma^\infty_+W)\rightarrow T$ satisfies $f(1) = 1$, and $s = \Phi^H t \in \pi_0 D(\Sigma^\infty_+W)$ is invertible, then $f(s) \in \pi_0 T$ is invertible.

If $H\notin\calF$, then $W = \emptyset$ and there is nothing to check. If $H\in\calF$, then there is an equivalence $W\simeq \bigvee S^{m-1}$. Choosing such an equivalence, the unit $S\rightarrow D(\Sigma^\infty_+ W)$ splits off giving
\[
D(\Sigma^\infty_+W)\simeq S \oplus \bigvee S^{-(m-1)}.
\]
The unit $s$ must appear in this splitting as $s = (\pm 1,\epsilon)$ for some $\epsilon \in \pi_0 \bigvee S^{-(m-1)}\cong \bigoplus \pi_{m-1}S$. As $f(1) = 1$, it follows that $f(s) = \pm 1 + g(\epsilon)$ for some $g\colon \bigvee S^{-(m-1)}\rightarrow T$. As $m\geq 2$, Nishida nilpotence \cite{nishida1973nilpotency} implies that $\epsilon$ is $\otimes$-nilpotent, and thus $g(\epsilon)$ is nilpotent. This implies that $f(s)$ is invertible as claimed.
\end{proof}

\begin{rmk}
The condition $m\geq 2$ in \cref{prop:nishida} is necessary. For example, let $G = C_2$ and write $\ul{S}^1 = S(2\sigma)$ for the $1$-sphere with its free antipodal action. Then $R = D(\Sigma^\infty_+ \ul{S}^1)[\tfrac{1}{2}]$ satisfies $\exp_{\{e\}}(R) = 1$, but $\pi_{1-\sigma}^{C_2}R = 0$ and so every map $s\colon D(C_{2+})\rightarrow R$ sends the invertible element $u_\sigma \in \pi_{1-\sigma}^{C_2}D(C_{2+})\cong \pi_0 S \cong \bbZ$ to zero. On the other hand, if $m=1$ then the $\epsilon$ appearing in the proof of \cref{prop:nishida} must be divisible by $2$, and so the proof goes through provided each $\Phi^HR$ is $2$-complete.
\tqed
\end{rmk}

\subsection{Existence of \texorpdfstring{$t_\alpha$}{t\_alpha}-elements}\label{ssec:maint}

Fix $\alpha \in RF(G)$ of virtual dimension zero. We can now formulate and prove our general existence theorems on $t_\alpha$-elements and $t_\alpha$-self maps. We first introduce a convenient definition.

\begin{defn}
The \textit{$\calF$-order} $\ord_\calF(\alpha)$ of $\alpha$ is the minimal $m$ for which $\alpha\neq 0$ in $KF_G^0(E\calF^{<m+1})$, with $\ord_\calF(\alpha) = \infty$ if no such $m$ exists.
\tqed
\end{defn}

In other words,
\[
m\leq\ord_\calF(\alpha)\quad \iff\quad \alpha = 0 \text{ in } KF_G^0(E\calF^{<m}),
\]
and $\ord_\calF(\alpha) = m$ precisely when $\alpha$ is detected on the $m$-line of the equivariant Atiyah--Hirzebruch spectral sequence
\[
H^\ast_G(E\calF;\ul{\pi}_\ast KF_G)\Rightarrow KF_G^\ast(E\calF).
\]
The most familiar case is when $\calF = \{e\}$, where $E\calF = EG$ and this is equivalent to the nonequivariant Atiyah--Hirzebruch spectral sequence for $KF^\ast BG$. Observe that $\ord_\calF(\alpha) \geq 0$, with equality if and only if $\alpha_H \neq 0$ for some $H\in\calF$. Thus if we define the family
\[
\calF(\alpha) = \{H\subset G: \alpha_H = 0\},
\]
then $\ord_\calF(\alpha)>0$ if and only if $\calF\subset\calF(\alpha)$. This refines to the following.

\begin{lemma}\label{lem:family}
If $\calF\subset\calF(\alpha)$, then the sequence $\{\ord_\calF(|G|^k\alpha):k\geq 0\}$ is unbounded.
\end{lemma}
\begin{proof}
The claim is that for all $m\geq 0$ there exists some $k\geq 0$ for which $m\leq \ord_\calF(|G|^k\alpha)$, that is for which $|G|^k\alpha = 0$ in $KF_G^0(E\calF^{<m})$. As $\calF\subset\calF(\alpha)$ we have $(E\calF^{<m})^H \neq \emptyset\Rightarrow \alpha_H = 0$, and so this follows from \cref{prop:mainlem}.
\end{proof}

\begin{lemma}\label{lem:characterfamily}
Suppose that $X$ is finite or a $G$-ring spectrum. Then $X$ is $\calF(\alpha)$-nilpotent if and only if $\Phi^C X \neq 0 \Rightarrow\alpha_C = 0$ for all cyclic subgroups $C\subset G$.
\end{lemma}
\begin{proof}
Suppose that $\Phi^C X \neq 0 \Rightarrow \alpha_C = 0$ for all cyclic subgroups $C\subset G$. As $X$ is finite or a $G$-ring spectrum, the collection $\{H\subset G : \Phi^H X \neq 0 \}$ forms a family of subgroups of $G$ closed under subconjugacy. It follows that $\Phi^H X \neq 0 \Rightarrow (C\subset H\text{ cyclic}\Rightarrow\alpha_C = 0)$. As a representation is determined by its restriction to cyclic subgroups we deduce $\Phi^H X \neq 0 \Rightarrow \alpha_H = 0$ for all $H\subset G$. Thus $X$ is $\calF(\alpha)$-nilpotent by \cref{lem:nildef}, which also gives the converse.
\end{proof}

We now give the main theorem of this section. Fix a family $\calF\subset \calF(\alpha)$.

\begin{theorem}\label{thm:maint}
Let $R$ be a $\calF$-nilpotent $G$-ring spectrum, and suppose that $\ord_\calF(n\alpha) \geq \max(2,\exp_\calF(R))$. Then $R$ admits a $t_\alpha$-element of order $n$. In particular, if $\Phi^C R \neq 0 \Rightarrow \alpha_C = 0$ for all cyclic subgroups $C\subset G$, then $R$ admits a $t_\alpha$-element of order dividing a power of $|G|$.
\end{theorem}
\begin{proof}
Abbreviate $m = \max(2,\exp_\calF(R))$.  As $m\leq \ord_\calF(n\alpha)$, we have $n\alpha = 0$ in $KF_G^0(E\calF^{<m})$. By \cref{cor:existk}, there is a $t_\alpha$-element $t \in \pi_{n\alpha}^GD(E\calF_+^{<m})$. As $m\geq \exp_\calF(R)$, there is a map $r\colon D(E\calF_+^{<m})\rightarrow R$ satisfying $r(1) = 1$. As $m\geq 2$, this preserves $t_\alpha$-elements by \cref{prop:nishida}. Thus $r(t) \in \pi_{n\alpha}^GR$ is a $t_\alpha$-element of order $n$. The final statement follows from \cref{lem:characterfamily}, which ensures that $R$ is $\calF(\alpha)$-nilpotent, and \cref{lem:family}, which ensures that $\ord_{\calF(\alpha)}(n\alpha)\geq \max(2,\exp_{\calF(\alpha)}(R))$ for some $n$ dividing a power of $|G|$.
\end{proof}

\begin{cor}\label{cor:maint}
Let $X$ be an $\calF$-nilpotent $G$-spectrum, and suppose $\ord_\calF(n\alpha) \geq \max(2,\exp_\calF(R))$. Then $X$ admits a $t_\alpha$-self map of order $n$. In particular, if $X$ is $\calF(\alpha)$-nilpotent then $X$ admits a $t_\alpha$-self map of order dividing a power of $|G|$; and if $X$ is compact, then it suffices to verify just that $\Phi^C X \neq 0 \Rightarrow\alpha_C = 0$ for all cyclic subgroups $C\subset G$.
\end{cor}
\begin{proof}
By \cite[Corollary 4.15]{mathewnaumannnoel2017nilpotence}, we have $\exp_\calF(X) = \exp_\calF(\End(X))$. As $X$ is an $\End(X)$-module, a $t_\alpha$-element in $\End(X)$ induces a $t_\alpha$-self map on $X$, so apply \cref{thm:maint} to $\End(X)$.
\end{proof}

\section{General and local equivalences \texorpdfstring{$\Sigma^V X \simeq \Sigma^W X$}{SigmaV X simeq SigmaW X}}\label{sec:j}

In this section we make some observations about equivalences $\Sigma^V X\simeq \Sigma^W X$, possibly after localization, which need not be $t_{V-W}$-self maps.

\subsection{Stable equivalences of representation spheres}\label{ssec:htpyequivreps}

We begin by summarizing work of tom Dieck, Petrie, and Tornehave on stable equivalences between representation spheres. Given a compact Lie group $G$, the representation rings $RU(G)$ and $RO(G)$ come equipped with \textit{Adams operations} $\psi^k$. These can be computed by their action on characters, given by
\[
\chi_{\psi^k V}(g) = \chi_V(g^k).
\]
If $G$ is finite then these operations satisfy $\psi^k = \psi^l$ when $k\equiv l\pmod{|G|}$, and thus the operations $\psi^k$ for $k$ coprime to $G$ induce an action of
\[
\Gamma = (\bbZ/|G|)^\times
\]
on $RU(G)$ and $RO(G)$. See \cite[II \S 3]{atiyahtall1969group} for a good discussion of this action (as well as for earlier work on stably equivalent representations). Write $I(\Gamma)$ for the augmentation ideal of $\bbZ[\Gamma]$, so that if $M$ is a $\Gamma$-module then $M/I(\Gamma)M$ is identified with the orbits $M_\Gamma$.

\begin{theorem}\label{thm:coprime}
If $G$ is finite and $k$ is coprime to $|G|$, then there exists a stable map
\[
f\colon S^V\rightarrow S^{\psi^k V}
\]
which is an equivalence after inverting $k$. Moreover, if $V$ is complex then one can choose $f$ to satisfy
\[
\Phi^H f =  k^{\lfloor |V^H|/2\rfloor}
\]
for all $H\subset G$.
\end{theorem}
\begin{proof}
The first statement follows from the proof of \cite[Theorem 10.12]{dieck1979transformation}. The refinement is due to Tornehave and appears in \cite[Theorem 4]{dieckpetrie1978geometric}. See also \cite[Theorem 4.1]{tornehave1982equivariant} for the more delicate case where $V$ is not assumed to be complex.
\end{proof}

\begin{ex}\label{ex:powermap}
Let $T \subset \bbC^\times$ denote the circle group and $L = S(\bbC)$ be the tautological complex character of $T$. Then the $k$th power map
\[
\psi_k\colon S(L^n)\rightarrow S(L^{nk})
\]
on unit spheres is $T$-equivariant. Passing to unreduced suspensions this yields a map
\[
\psi_k\colon S^{L^n}\rightarrow S^{L^{nk}}
\]
with the property that
\[
\Phi^{C_d}\psi_k = \begin{cases}
k,&d\mid n,\\
0,&d\mid nk\text{ but }d\nmid n,\\
1,&d\nmid n.
\end{cases}
\]
In particular, if $m$ is coprime to $k$ then $\psi_k$ is an equivalence after restricting to $C_m\subset T$ and inverting $k$.
\tqed
\end{ex}

The following theorem now summarizes information about about stably equivalent representation spheres. 

Write $\Pic(A(G))$ for the Picard group of $A(G) = \pi_0 S_G$. For a spectrum or abelian group $M$ and integer $n$, write $M_{(n)} = M[p^{-1}:\gcd(p,n) = 1]$.

\begin{theorem}\label{thm:tdp}
Let $G$ be a compact Lie group and $\alpha \in RO(G)$. The following are equivalent:
\begin{enumerate}
\item $|\alpha^H| = 0$ for all $H\subset G$.
\item $\pi_\alpha S_G \in \Pic(A(G))$, with
\[
\pi_{\alpha+\star}X \cong \pi_\alpha S_G \otimes_{A(G)} \pi_\star X
\]
for any $G$-spectrum $X$.
\item There exists an equivalence $S^{n\alpha}\simeq S^0$ for some $n\geq 1$.
\end{enumerate}
If $G$ is finite, then these are moreover equivalent to the following:
\begin{enumerate}[resume]
\item $|\alpha^C| = 0$ for all cyclic $C\subset G$.
\item $\alpha \in I(\Gamma)\cdot RO(G)$.
\item There exists an equivalence $S^\alpha_{(|G|)}\simeq S_{(|G|)}$.
\end{enumerate}
In addition,
\begin{enumerate}[resume]
\item If $\alpha \in I(\Gamma)^2\cdot RO(G)$ then $S^\alpha\simeq S^0$, and
\item The converse holds if $G$ is a $p$-group.
\end{enumerate}
\end{theorem}
\begin{proof}
(1)$\Rightarrow$(2): This is \cite[Theorem 1]{dieckpetrie1978geometric}.

(2)$\Rightarrow$(3): The picard group $\Pic(A(G))$ has finite exponent \cite[Equation 32]{dieckpetrie1978geometric}, and thus
\[
A(G)\cong \pi_0 S_G\cong (\pi_\alpha S_G)^{\otimes_\alpha n}\cong \pi_{n\alpha}S_G
\]
for some $n\geq 1$. The image of $1$ under such an isomorphism is an invertible element in $\pi_{n\alpha}S_G$, giving an equivalence $S^{n\alpha}\simeq S^0$.

(3)$\Rightarrow$(1): If $S^{n\alpha}\simeq S^0$ for some $n\geq 1$, then applying $\Phi^H$ implies $S^{n|\alpha^H|}\simeq S^0$, and thus $|\alpha^H| = 0$ for all $H\subset G$.

(1)$\Leftrightarrow$(4)$\Leftrightarrow$(5): This is \cite[Proposition 9.2.6]{dieck1979transformation}.

(5)$\Rightarrow$(6): This follows from \cref{thm:coprime}.

(6)$\Rightarrow$(1): Same proof as (3)$\Rightarrow$(1).

(7,8): These are \cite[Theorems 9.1.4, 9.1.5]{dieck1979transformation}.
\end{proof}

We end this subsection with some comments on how these results can be applied in practice. Say that two virtual representations $\alpha$ and $\alpha'$ are \textit{locally $J$-equivalent} if $\alpha-\alpha' \in I(\Gamma)\cdot RO(G)$. \cref{thm:tdp} says that this is equivalent to the existence of a unit $p_{\alpha'-\alpha}\in \pi_{\alpha'-\alpha}(S_G)_{(|G|)}$. If one finds an algebraic witness to $\alpha'-\alpha \in I(\Gamma)\cdot RO(G)$, for example if $\alpha' = \psi^k\alpha$, then \cref{thm:coprime} gives some control over the behavior of (a choice for) $p_{\alpha'-\alpha}$.  Moreover, in many basic cases \cref{ex:powermap} and variations are already sufficient and give completely explicit choices of $p_{\alpha'-\alpha}$.

\begin{ex}
If $Z$ is a compact $G$-space and $\alpha$ vanishes in $KO_G^0(Z)$, then $\alpha$ lifts to $b \in \widetilde{KO}{}_G^0(SZ)$ and \cref{thm:jhom} produces a unit $t_\alpha = J(b) \in \pi_\alpha D(\Sigma^\infty_+ Z)$ with good properties: for example, $\res^G_e(t_\alpha) = J(\res^G_e(b)) \in \pi_0^e D(\Sigma^\infty_+ Z)$ is in the image of the classical $J$-homomorphism. Thus if $\alpha'$ is locally $J$-equivalent to $\alpha$ then one obtains a unit
\[
p_{\alpha'-\alpha} \cdot t_\alpha \in \pi_{\alpha'}^GD(\Sigma^\infty_+Z)_{(|G|)}
\]
with similarly good properties: for example, $\res^G_e(p_{\alpha'-\alpha}\cdot J(b)) = k\cdot J(\res^G_e(b))$ for some integer $k$ coprime to $|G|$.
\tqed
\end{ex}

\begin{ex}\label{ex:jdiff}
In the situation of \cref{thm:hfpss}, if $\alpha'$ is locally $J$-equivalent to $\alpha$ then one also has
\[
d_f(t_{\alpha'}) = \pm \tilde{J}(b) t_{\alpha'}
\]
in the homotopy fixed point spectral sequence. Indeed, $p_{\alpha'-\alpha} \in \pi_{\alpha'-\alpha}^G(S_G)_{(|G|)}$ is detected by $k\cdot t_{\alpha'}t_{\alpha}^{-1}$ for some integer $k$ coprime to $|G|$, which must then be a permanent cycle. As the spectral sequence is $|G|$-torsion in positive filtration, it follows that $t_{\alpha'}t_{\alpha}^{-1}$ is a permanent cycle. Thus
\[
d_f(t_{\alpha'}) = d_f(t_{\alpha'}t_\alpha^{-1}t_{\alpha}) = t_{\alpha'}t_\alpha^{-1}\cdot d_f(t_\alpha) = t_{\alpha'}t_{\alpha}^{-1}\cdot \pm \tilde{J}(b) t_\alpha = \pm \tilde{J}(b) t_{\alpha'}
\]
as claimed. Similarly considerations hold for the more general equivariant Atiyah--Hirzebruch spectral sequences handled in \cref{sec:ahss}.
\tqed
\end{ex}

One might ask to what extent these techniques account for everything, and to that end we leave the following question.

\begin{question}
Let $G$ be a finite group and $Z$ be a compact $G$-space. Is every unit in $\pi_\star^G D(\Sigma^\infty_+ Z)_{(|G|)}$ of the form $c\cdot J(b)$, where $b\in \widetilde{KO}{}_G^0(SZ)$ and $c$ lifts to $\pi_\star (S_G)_{(|G|)}^\times$?
\tqed
\end{question}

\subsection{General periodicities}

Throughout this subsection $G$ is a finite group. We now explain how \cref{thm:tdp} implies a general theorem about the eventual existence of equivalences $\Sigma^{n\alpha}X\simeq X$.

\begin{lemma}\label{lem:1dim}
Let $Z$ be a $1$-dimensional $G$-complex. If $Z^H \neq \emptyset \Rightarrow |\alpha^H| = 0$ for all subgroups $H\subset G$, then there is an invertible element $t \in \pi_{n\alpha}^GD(\Sigma^\infty_+Z)$ for some $n\geq 1$.
\end{lemma}
\begin{proof}
As $Z$ is $1$-dimensional, it can be built as a homotopy coequalizer of the form
\begin{center}\begin{tikzcd}
\coprod_{i\in I}G/H_i\ar[r,"f",shift left]\ar[r,"g"',shift right]&\coprod_{j\in J}G/H_j\ar[r]&Z
\end{tikzcd}.\end{center}
It follows that $D(\Sigma^\infty_+Z)$ is an equalizer of the form
\begin{center}\begin{tikzcd}
D(\Sigma^\infty_+Z)\ar[r]&\prod_{j\in J}D(G/H_{j+})\ar[r,"f^\ast",shift left]\ar[r,"g^\ast"',shift right]&\prod_{i\in I}D(G/H_{i+})
\end{tikzcd}.\end{center}
Note that in general $\pi_\alpha^G D(G/H_+) \cong \pi_{\alpha_H} S_H$. Applying \cref{thm:tdp}(1$\Rightarrow$3) to the restrictions $\alpha_{H_j}$, as $G$ is finite we can find some $n\geq 1$ for which there exist equivalences $u_j\colon S^{n\alpha}_{H_j}\simeq S^0_{H_j}$. These determine an invertible element $u \in \pi_{n\alpha}^G \prod_{j\in J}D(G/H_{j+})$. Our conventions from \cref{rmk:signs} are such that $\pi_{n\alpha}^G$ is only defined up to a sign, but this issue goes away after passing to $u^2 \in \pi_{\bbC\otimes n\alpha}^G\prod_{j\in J}D(G/H_{j+})$. Moreover, this square is guaranteed to satisfy $f^\ast(u^2) = g^\ast(u^2)$, so $u^2$ lifts to an invertible element $t\in \pi_{2n\alpha}^G D(\Sigma^\infty_+ Z)$. 
\end{proof}

\begin{theorem}\label{thm:jglob}
Let $G$ be a finite group and $R$ be a $G$-ring spectrum, and suppose that $\Phi^C R \neq 0 \Rightarrow |\alpha^C| = 0$ for all cyclic subgroups $C\subset G$. Then there exists an invertible element $t\in \pi_{n\alpha}^GR$ for some $n\geq 1$. The converse holds if each $\Phi^C R$ is bounded below.
\end{theorem}
\begin{proof}
Define the family
\[
\calF[\alpha] = \{H\subset G : C\subset H \text{ cyclic}\Rightarrow |\alpha^C|=0\},
\]
and consider the Atiyah--Hirzebruch spectral sequence
\[
E_2 = H^\ast_G(E\calF[\alpha];\ul{\pi}_\star R)\Rightarrow R^{\ast-\star}_G(E\calF[\alpha]).
\]
The argument in \cref{lem:characterfamily} adapts, in conjunction with \cref{thm:tdp}(1$\Leftrightarrow$4), to show that $R$ is $\calF[\alpha]$-nilpotent, so this converges to $\pi_\star^G R$ with a horizontal vanishing line at a finite page. \cref{lem:1dim} ensures that there exists an invertible class $u\in H^0_G(\calF;\ul{\pi}_{k\alpha}R)$ for some $k\geq 1$. A theorem of Dress shows that $H^{>0}_G(E\calF;\ul{\pi}_\star R)$ is killed by $|G|$ \cite[Proposition 21.3]{greenleesmay1995generalized}. Thus the Leibniz rule implies that if $u^i$ survives to the $E_r$-page then $u^{i|G|}$ survives to the $E_{r+1}$-page. By the horizontal vanishing line it follows that some power $u^{|G|^m}$ is a permanent cycle, and we can take $t$ to be any class detected by this power.

Conversely, if $t\in \pi_{n\alpha}^GR$ is invertible, then $\Phi^C t \in \pi_{n|\alpha^C|}\Phi^C R$ is invertible. As $\Phi^CR$ is bounded below, this is only possible if either $\Phi^C R = 0$ or $|\alpha^C| = 0$ as claimed.
\end{proof}

We deduce \cref{thm:generalequiv} as a corollary.

\begin{theorem}\label{cor:jg}
Let $G$ be a finite group and $X$ be a compact $G$-spectrum. Then there is an equivalence $\Sigma^{n\alpha}X\simeq X$ for some $n\geq 1$ if and only if $\Phi^C X \neq 0 \Rightarrow |\alpha^C| = 0$ for all cyclic subgroups $C\subset G$.
\end{theorem}
\begin{proof}
Apply \cref{thm:jglob} to the $G$-ring spectrum $\End(X)$.
\end{proof}

\subsection{The equivariant Adams conjecture}

Throughout this subsection $G$ is a finite group and $Z$ is a compact $G$-space. We now use work of tom Dieck and Hauschild \cite[Chapter 11]{dieck1979transformation} on the equivariant Adams conjecture to obtain more information about units in $\pi_\star^G D(\Sigma^\infty_+ Z)_{(|G|)}$. See also work of McClure \cite{mcclure1983on}. The starting point is the following. Say that a stable map $f\colon S(\xi)\rightarrow S(\zeta)$ of sphere bundles over $Z$ has \textit{fiberwise degree dividing a power of $k$} if for all $H\subset G$ and $z\in Z^H$, the induced map $f^H_z\colon S(\xi^H)_z\rightarrow S(\zeta^H)_z$ on fibers has degree dividing a power of $k$.

\begin{theorem}[{\cite[Theorem 11.3.8, Proposition 11.4.4]{dieck1979transformation}}]\label{thm:adamsconj}
Let $\xi$ be a vector bundle over $Z$ and $k$ be an odd positive integer coprime to $|G|$. Then there exist stable maps $f\colon S(\xi)\rightarrow S(\psi^k\xi)$ and $g\colon S(\psi^k \xi)\rightarrow S(\xi)$ of fiberwise degree dividing a power of $k$.
\qed
\end{theorem}

The Thom spectrum $\Th(\xi)$ of a vector bundle $\xi$ depends only on its associated sphere bundle $S(\xi)$, and is functorial in stable maps. The condition that a map $f\colon S(\xi)\rightarrow S(\zeta)$ has fiberwise degree dividing a power of $k$ ensures that it induces an equivalence $\Th(\xi)[\tfrac{1}{k}]\rightarrow\Th(\zeta)[\tfrac{1}{k}]$. Hence if we define
\[
J_G^{\alg}(Z) = KO_G^0(Z)/(x-\psi^k x :x\in KO_G^0(Z),\, 2\nmid k,\, \gcd(k,|G|) = 1),
\]
and write $j\colon KO_G^0(Z)\rightarrow J_G^\alg(Z)$ for the quotient map, then we have the following.

\begin{prop}\label{lem:jper}
If $j(\alpha) = 0$ in $J_G^\alg(Z)$, then there is an invertible element in $\pi_\alpha^G D(\Sigma^\infty_+Z)_{(|G|)}$.
\end{prop}
\begin{proof}
This follows from \cref{thm:adamsconj}, using the same considerations as in \cref{ssec:k} to translate between stable maps $S(Z_V)\rightarrow S(Z_W)$ and elements in $\pi_{V-W}^GD(\Sigma^\infty_+ Z)$.
\end{proof}

Our main goal in the rest of this section is to prove a converse to \cref{lem:jper} when $G$ is a $p$-group.

\begin{lemma}\label{lem:dest}
Write $\alpha = V-W$ as a difference of representations. If there is an invertible class in $\pi_\alpha^G D(\Sigma^\infty_+Z)[\tfrac{1}{k}]$, then there is a stable map $f\colon S(Z_V)\rightarrow S(Z_W)$ of fiberwise degree dividing a power of $k$.
\end{lemma}
\begin{proof}
We may suppose ourselves given a map $\phi\colon \Sigma^V\Sigma^\infty_+Z\rightarrow S^W$ associated to an element $u\in \pi_\alpha^G D(\Sigma^\infty_+Z)$ which becomes invertible after inverting $k$. As $Z$ is compact, after possibly enlarging $V$ and $W$ we may write $\phi$ as the stabilization of a map
\[
f\colon \Sigma^V(Z_+)\rightarrow S^W.
\]
The assumption that $u$ is invertible after inverting $k$ ensures that for all $H\subset G$ and $z\in Z^H$, the induced map
\[
f_z^H\colon S^{V^H}\rightarrow S^{W^H},\qquad f_z^H(v) = f(v\wedge z)
\]
has degree dividing a power of $k$. As $S^V \cong S(V+1)$, it follows that
\[
\tilde{f}\colon S^V\times Z\rightarrow S^W\times Z,\qquad \tilde{f}(v,z) = (f(v\wedge z),z)
\]
defines a map $S(Z_{V+1})\rightarrow S(Z_{W+1})$, hence a stable map $S(Z_V)\rightarrow S(Z_W)$, of fiberwise degree dividing a power of $k$.
\end{proof}

\begin{lemma}\label{lem:deloc}
If $j(\alpha) = 0$ in $J^\alg_G(Z)_{(|G|)}$, then $j(\alpha)=0$ in $J^\alg_G(Z)$.
\end{lemma}
\begin{proof}
Deferred to the next subsection, where it appears as \cref{cor:zc}.
\end{proof}

Now fix a prime $p$, and suppose that $G$ is a $p$-group. As discussed in the introduction, the $G$-spectrum $KU_G/p$ plays a similar role in $G$-equivariant homotopy theory as $KU/p$ does in nonequivariant homotopy theory. For example, if $\ell$ generates a dense subgroup of $\bbZ_p^\times/\{\pm 1\}$---and we may as well assume $\ell$ is odd here---and we define
\[
J_G = \Fib\left(\psi^\ell-\psi^1\colon (KO_G)_p^\wedge\rightarrow (KO_G)_p^\wedge\right),
\]
then for any compact $G$-space $Z$ there is an equivalence
\[
L_{KU_G/p}D(\Sigma^\infty_+ Z)\simeq F(\Sigma^\infty_+ Z,J_G),
\]
where $L_{KU_G/p}$ denotes Bousfield localization with respect to $KU_G/p$. Write
\[
j_{K(1)}^Z\colon RO(G)\rightarrow  KO_G^0(Z)\rightarrow J_G^1(Z)
\]
for the resulting boundary map. We can now give the following.

\begin{theorem}\label{thm:k1}
Let $Z$ be a compact $G$-space and $\alpha \in RO(G)$. Then there exists an invertible element in $\pi_\alpha^G D(\Sigma^\infty_+ Z)_{(p)}$ if and only if $j_{K(1)}^Z(\alpha) = 0$.
\end{theorem}
\begin{proof}
If there exists an invertible element in $\pi_\alpha^G D(\Sigma^\infty_+ Z)_{(p)}$, then after writing $\alpha = V-W$ as a difference of representations, \cref{lem:dest} provides a stable map $S(Z_V)\rightarrow S(Z_W)$ of fiberwise degree coprime to $p$. Now \cite[Theorem 11.4.1, Proposition 11.4.2]{dieck1979transformation} implies that $j_{K(1)}^Z(\alpha) = 0$ in $J_G^1(Z)$.

Conversely, if $j_{K(1)}^Z(\alpha) = 0$, then as $p$-completion is faithful for finitely generated $\bbZ_{(p)}$-modules we can decomplete to say that $\alpha$ is sent to zero in $\coker(\psi^\ell-\psi^1\colon KO_G^0(Z)_{(p)}\rightarrow KO_G^0(Z)_{(p)})$, and thus also in $J_G^{\alg}(Z)_{(p)}$. By \cref{lem:deloc} we deduce $j(\alpha) = 0$ in $J_G^{\alg}(Z)$ and thus there is an invertible element in $\pi_\alpha^G D(\Sigma^\infty_+Z )_{(p)}$ by \cref{lem:jper}.
\end{proof}

\subsection{Localization arguments}

We now make good on \cref{lem:deloc}. Given odd positive integers $\vec{r} = (r_1,\ldots,r_t)$ coprime to $|G|$, set $r = r_1\cdots r_t$ and define
\[
J^\alg_{G,\vec{r}}(Z) = KO_G^0(Z)[\tfrac{1}{r}]/(x-\psi^{r_i}x:x\in KO_G^0(Z),\, 1\leq i \leq r).
\]
As $\psi^{r_i}\colon KO_G^0(Z)[\tfrac{1}{r}]\rightarrow KO_G^0(Z)[\tfrac{1}{r}]$ is a stable operation, it commutes with restrictions and transfers, and so $J^\alg_{G,\vec{r}}(Z)$ is a quotient Mackey functor of $KO_G^0(Z)[\tfrac{1}{r}]$. It relates to $J^\alg_G(Z)$ via a commutative diagram
\begin{center}\begin{tikzcd}
KO_G^0(Z)\ar[d,"j_{\vec{r}}"]\ar[r,"j"]&J^\alg_G(Z)\ar[d]\\
J^\alg_{G,\vec{r}} (Z)\ar[r,two heads]&J^\alg_G(Z)[\tfrac{1}{r}]
\end{tikzcd}.\end{center}
If $k\equiv l \pmod{|G|}$ then $\psi^k=\psi^l$ on $RO(G)$, and it follows that the bottom surjection is an isomorphism for $Z = \ast$ provided that $\vec{r}$ generates $(\bbZ/|G|)^\times$.

\begin{lemma}\label{lem:lemgen}
Suppose that $\vec{r}$ generates $(\bbZ/|G|)^\times$ and $Z^C\neq\emptyset\Rightarrow |\alpha^C|=0$ for all cyclic $C\subset G$. Then $j(\alpha) \in J^\alg_{G,\vec{r}} (Z)$ is $|G|$-power torsion.
\end{lemma}
\begin{proof}
The claim is that $j(\alpha) = 0$ in $J^\alg_{G,\vec{r}} (Z)[\tfrac{1}{|G|}]$. As $J^\alg_{\bs,\vec{r}} (Z)$ is a Mackey functor, we may use \cref{lem:macksplit} to reduce to showing that if $Z^H \neq \emptyset$ then $j_{H,\vec{r}}(\alpha) = 0$. 

So suppose $Z^H \neq \emptyset$. It follows that $Z^C\neq\emptyset$ for all cyclic subgroups $C\subset H$. Thus $|\alpha^C|=0$ for all cyclic subgroups $C\subset H$ by assumption, and so $j_H(\alpha) = 0$ in $J^\alg_H(\ast)$ by \cref{thm:tdp}(1$\Leftrightarrow$4). As $\vec{r}$ generates $(\bbZ/|G|)^\times$ we have $J^\alg_H(\ast)[\tfrac{1}{r}]\cong J^\alg_{H,\vec{r}}(\ast)$, implying $j_{H,\vec{r}}(\alpha) = 0$ in $J^\alg_{H,\vec{r}} (\ast)$. Pulling back along $Z\rightarrow\ast$ it follows that $j_{H,\vec{r}}(\alpha)=0$ in $J^\alg_{H,\vec{r}} (Z)$.
\end{proof}

\begin{prop}\label{prop:zc}
The class $j(\alpha) \in J^\alg_G(Z)$ has finite order if and only if $Z^C\neq\emptyset\Rightarrow|\alpha^C|=0$ for all cyclic subgroups $C\subset G$. In this case, its order divides a power of $|G|$.
\end{prop}
\begin{proof}
First suppose that $Z^C\neq\emptyset\Rightarrow|\alpha^C|=0$ for all cyclic subgroups $C\subset G$. By \cref{lem:lemgen} and the comparison map $J^\alg_{G,\vec{r}} (Z)\rightarrow J^\alg_G(Z)[\tfrac{1}{r}]$, we find that if $\vec{r}$ generates $(\bbZ/|G|)^\times$ then $j(\alpha) = 0$ in $J^\alg_G(Z)[\tfrac{1}{r|G|}]$. So fix $\vec{r}$ generating $(\bbZ/|G|)^\times$ and $\vec{s}$ generating $(\bbZ/r|G|)^\times$. Then $\vec{s}$ also generates $(\bbZ/|G|)^\times$, so $j(\alpha) = 0$ in $J^\alg_G(Z)[\tfrac{1}{r|G|}]$ and $J^\alg_G(Z)[\tfrac{1}{s|G|}]$, and $\gcd(r,s) = 1$ then implies $j(\alpha) = 0$ in $J^\alg_G(Z)[\tfrac{1}{|G|}]$. Thus $j(\alpha)$ has finite order dividing a power of $|G|$.

Conversely, if $Z^C\neq\emptyset$, then there is some equivariant map $p\colon G/C\rightarrow Z$. This must satisfy $p^\ast(j(\alpha)) = j_C(\alpha) \in J^\alg_G(G/C)\cong J^\alg_C(\ast)$. It follows that if $j(\alpha)$ is torsion, then so is $j_C(\alpha)$, implying that $j_C(\alpha) = 0$ as $J^\alg_C(\ast)$ is torsion-free. Thus $|\alpha^C| = 0$ by \cref{thm:tdp}.
\end{proof}

\begin{cor}\label{cor:zc}
If $j(\alpha) = 0$ in $J^\alg_G(Z)_{(|G|)}$, then $j(\alpha)=0$ in $J^\alg_G(Z)$.
\end{cor}
\begin{proof}
If $j(\alpha) = 0$ in $J^\alg_G(Z)_{(|G|)}$, then $j(\alpha) \in J^\alg_G(Z)$ has finite order coprime to $|G|$. \cref{prop:zc} then implies that $j(\alpha)$ also has order dividing a power of $|G|$, so $j(\alpha)=0$.
\end{proof}

At this point, one could also carry out a $J$-theoretic analogue of \cref{ssec:maint}, giving information about invertible elements in $|G|$-local $G$-ring spectra. We leave the details to the interested reader.

\section{Examples}\label{sec:examples}

In this section we give examples of the material of the previous sections, focusing especially on the $J$-homomorphism
\[
\pi_\lambda KO_G\rightarrow \pi_\star C(a_\lambda)^\times
\]
derived from \cref{thm:jhom} and the equivalence $D(\Sigma^\infty_+ S(\lambda))\simeq C(a_\lambda)$ guaranteed by the cofiber sequence $S(\lambda)_+\rightarrow S^0\rightarrow S^\lambda$. Our goal is to demonstrate that this is in fact quite computable, and that it produces explicit computational information about $G$-equivariant homotopy theory and $G$-equivariant stable stems. 

The bulk of our work in this section lies in computing the groups $\pi_\star KO_G$, particularly information about $a_\lambda\colon \pi_\lambda KO_G\rightarrow RO(G)$. We discuss this in general \cref{ssec:kcompute}, and the examples we give have been chosen to illustrate such computations. We focus on the case where $G$ is finite.

In our examples we write
\[
j_n \in \pi_n S
\]
for the $J$-image of a generator of $\pi_{n+1}KO$, defined up to a sign, with the understanding that $j_0 = \pm 2$.

\subsection{Computing with equivariant \texorpdfstring{$K$}{K}-theory}\label{ssec:kcompute}

Let $G$ be a compact Lie group and $\lambda$ be a $G$-representation. To apply our machinery to produce  periodicities on $a_\lambda$-torsion, one needs to be able to understand the groups $\pi_{\ast\lambda}KF_G$:

\begin{center}\begin{tikzcd}
RF(G) = \pi_0 KF_G&\ar[l,"a_\lambda"']\pi_\lambda KF_G&\ar[l,"a_\lambda"']\pi_{2\lambda}KF_G&\ar[l,"a_\lambda"']\cdots
\end{tikzcd}.\end{center}
When $F = \bbC$ and $\lambda$ is a complex representation, equivariant Bott periodicity implies
\[
\pi_\lambda KU_G \cong RU(G)\{\beta_\lambda\},\qquad a_\lambda\beta_\lambda = e_\lambda,
\]
where if $V$ is a complex $G$-representation then we write
\[
e_V = \sum_i (-1)^i \Lambda^i V \in RU(G)
\]
for the $K$-theory Euler class of $V$, see for example \cite[Section IV.1]{atiyahtall1969group}. This Euler class can be computed using character information, for example combining the character identity $\chi_{\psi^k V}(g)=\chi_V(g^k)$ with Newton's identity $k\Lambda^k V = \sum_{i=1}^k (-1)^{i-1}\psi^iV\cdot \Lambda^{k-i}V$. It can also be computed using representation information: if $p_V(g,t)$ is the characteristic polynomial of $g\colon V\rightarrow V$, then $\chi_{e_V}(g) = p_V(g,1)$.

In general we do not have a complete recipe for $\pi_{\ast\lambda}KU_G$ when $\lambda$ does not admit a complex structure. As one always has Bott periodicity in the form $\pi_{(\ast+2)\lambda}KU_G\cong \pi_{\ast\lambda}KU_G\{\beta_{\bbC\otimes\lambda}\}$, we are left with the following problem.

\begin{samepage}
\begin{problem}\label{problem:roku}
For a real $G$-representation $\lambda$, describe the sequence
\begin{center}\begin{tikzcd}
RU(G)\cong \pi_0 KU_G&\ar[l,"a_\lambda"']\pi_\lambda KU_G&\ar[l,"a_\lambda"']\pi_{2\lambda}KU_G\cong RU(G)\{\beta_{\bbC\otimes\lambda}\}
\end{tikzcd},\end{center}
the composite of which is multiplication by $e_{\bbC\otimes\lambda}$.
\tqed
\end{problem}
\end{samepage}

\begin{rmk}\label{rmk:chern}
If $G$ is finite, then for any $\alpha \in RO(G)$, there is a natural character isomorphism
\[
\bbC \otimes \pi_\alpha KU_G \cong \prod_{\langle g \rangle}\widetilde{H}^0(S^{\alpha^g}/C(g),\bbC[\beta^{\pm 1}]).
\]
Here, the product is over the conjugacy classes of elements $g\in G$, and $C(g)$ is the centralizer of $g$ acting on the fixed points $S^{\alpha^g}$. In particular, the sequence of \cref{problem:roku} is easily understood after complexification. 
\tqed
\end{rmk}

\begin{rmk}
Karoubi \cite{karoubi2002equivariant} has shown that $\pi_\alpha KU_G$ is a free abelian group for any $\alpha\in RO(G)$. The ranks of these free abelian groups are determined by \cref{rmk:chern}, as described in \cite[Theorem 1.8]{karoubi2002equivariant}, so this completely describes $\pi_\star KU_G$ additively.
\tqed
\end{rmk}

Once enough is known about $\pi_\star KU_G$, one can descend to $\pi_\star KO_G$ using the homotopy fixed point spectral sequence (HFPSS)
\[
E_2 = H^\ast(C_2;\pi_\star KU_G)\Rightarrow \pi_{\star-\ast}KO_G,
\]
where $C_2 = \{\psi^{\pm 1}\}$ acts on $KU_G$ by complex conjugation. This also describes $\pi_\star KSp_G$ as $KSp_G\simeq \Sigma^4 KO_G$. We make some obervations about this spectral sequence.

\begin{rmk}\label{rmk:zkog}
The structure of $KO^G$ was determined by Segal \cite{segal1968equivariant}, and this in turn determines the HFPSS for $\pi_\ast KO_G$ in integer degrees. See \cite[Section 9]{mathewnaumannnoel2017nilpotence} for a detailed discussion. In particular, $KU^G$ is a free $KU$-module with basis indexed by the irreducible complex $G$-representations, and likewise $KO^G$ splits into a sum of $KO$-modules indexed by the irreducible real $G$-representations, where these summands are of the form $KO$, $KU$, or $KSp$, corresponding to the orthogonal, complex, and symplectic irreducibles.
\tqed
\end{rmk}

\begin{rmk}
The symplectic orientation of $KO$ implies that if $V$ is a quaternionic $G$-representation, then the associated $KU$-Thom class in $\pi_{V-|V|}KU_G$ descends to $KO_G$. See \cite[Section 5]{french2003equivariant} for further discussion.  In this case, $F(S^V,KO_G)\simeq \Sigma^{|V|}KO_G$ has fixed points determined by the representation theory of $G$ as in \cref{rmk:zkog}. If for example $V$ is a complex $G$-representation (such as $\bbC\otimes_\bbR U$ for a real representation $U$), then $V + \psi^{-1} V = \bbH \otimes_\bbC V$ is a quaternionic representation, so this reduces the computation of $\pi_\star KO_G$ to essentially finitely many degrees.
\tqed
\end{rmk}

Now suppose that $G$ is finite. We can further cut down the amount of work needed to understand $\pi_\star KF_G$ as follows. For $F = \bbR$ or $\bbC$, write 
\[
\epsilon\colon KF^G\rightarrow KF,\qquad \epsilon\colon RF(G)\rightarrow\bbZ,\qquad \epsilon(V) = \dim_F(V/G)
\]
for the projection onto the summand corresponding to the trivial representation.

\begin{lemma}\label{lem:frobenius}
For $F = \bbR$ or $\bbC$, the composite
\[
\langle\bs,\bbs\rangle\colon \begin{tikzcd}KF^G\otimes_{KF} KF^G\ar[r,"\mu"]&KF^G\ar[r,"\epsilon"]&KF\end{tikzcd}
\]
is adjoint to an equivalence $KF^G\simeq \Mod_{KF}(KF^G,KF)$.
\end{lemma}
\begin{proof}
As base change $\Mod_{KO}\rightarrow\Mod_{KU}$ is conservative, satisfies $KU\otimes_{KO}(KO^G)\simeq KU^G$, and is compatible with $\epsilon$, it suffices to consider the case $F = \bbC$. As $KU^G$ is free over $KU$, it suffices to show that $\langle \bs,\bbs\rangle$ induces a perfect pairing  $RU(G)\otimes_\bbZ RU(G)\rightarrow\bbZ$ on $\pi_0$. This pairing acts on irreducible $G$-representations $V$ and $W$ by
\[
\langle V,W\rangle = \dim_\bbC((V\otimes W)/G) = \dim_\bbC(\Hom(V^\vee,W)^G) = \begin{cases}1,&W\cong \psi^{-1}V,\\ 0,&\text{ otherwise}.\end{cases}
\]
This is evidently a perfect pairing, though note that it is not quite the usual inner product on $RU(G)$.
\end{proof}

This is an integral version of the height $1$ case of the $K(n)$-local duality considered by Strickland in \cite{strickland2000kn}. It has the following consequence.

\begin{prop}\label{prop:duality}
Let $F = \bbR$ or $\bbC$ and let $M$ be a $KF_G$-module. Then there is an equivalence
\begin{gather*}
\Mod_{KF_G}(M,KF_G)\rightsim \Mod_{KF}(M^G,KF),\\
(f\colon M\rightarrow KF_G)\mapsto (\epsilon\circ f^G\colon M^G\rightarrow KF^G\rightarrow KF)
\end{gather*}
of $KF$-modules.
\end{prop}
\begin{proof}
This is a natural transformation
\[
\Mod_{KF_G}(\bs,KF_G)\rightarrow \Mod_{KF}((\bs)^G,KF)
\]
between limit-preserving functors $\Mod_{KF_G}^\op\rightarrow \Mod_{KF}$. It therefore suffices to check that it is an equivalence when evaluated on $G/H_+ \otimes KF_G$ for $H\subset G$. But in this case the map is exactly the equivalence
\[
KF^H\simeq\Mod_{KF}(KF^H,KF)
\]
guaranteed by \cref{lem:frobenius}.
\end{proof}

\begin{cor}\label{cor:duality}
For any $\alpha\in RO(G)$, the map
\[
\langle\bs,\bbs\rangle \colon \pi_\alpha KU_G\otimes_\bbZ \pi_{-\alpha}KU_G\rightarrow \bbZ,\qquad \langle x,y\rangle = \epsilon(x y)
\]
is adjoint to an isomorphism $\pi_{\alpha}KU_G\cong\Hom_\bbZ(\pi_{-\alpha} KU_G,\bbZ)$. Moreover, if $\alpha=\lambda$ is a $G$-representation then $a_\lambda\colon \pi_0 KU_G\rightarrow \pi_{-\lambda}KU_G$ is dual to $a_\lambda\colon \pi_\lambda KU_G\rightarrow \pi_0 KU_G$.
\end{cor}
\begin{proof}
The adjoint $\pi_{\alpha}KU_G\rightarrow\Hom_\bbZ(\pi_{-\alpha} KU_G,\bbZ)$ may be written as
\begin{align*}
\pi_{\alpha}KU_G&\cong \pi_0 (\Mod_{KU_G}(KU_G\otimes S^\alpha,KU_G))\\
&\rightsim \pi_0 \Mod_{KU}((KU_G\otimes S^\alpha)^G,KU)\rightsim \Hom_\bbZ(\pi_{-\alpha}KU_G,\bbZ).
\end{align*}
Here, the first equivalence holds by definition, the second is obtained from \cref{prop:duality}, and the third holds as $(KU_G\otimes S^\alpha)^G$ is a free $KU$-module by Karoubi \cite{karoubi2002equivariant}. The final statement holds just as $a_\lambda\colon S^0\rightarrow S^\lambda$ is dual to $a_\lambda\colon S^{-\lambda}\rightarrow S^0$.
\end{proof}

\begin{rmk}\label{rmk:dualityhfpss}
By the Leibniz rule we have 
\[
0 = d_r(\langle x,y\rangle) = \langle d_r(x),y\rangle + \langle x,d_r(y)\rangle.
\]
In particular, the HFPSS for $\pi_{\ast+\alpha}KO_G$ is determined by the HFPSS for $\pi_{\ast-\alpha}KO_G$, cutting the work needed to compute $\pi_\star KO_G$ in half.
\tqed
\end{rmk}

\subsection{Example: cyclic groups}\label{ssex:cyclic}

Character theory ensures that much of the structure of $G$-equivariant $K$-theory is controlled by the case where $G$ is a cyclic group. So before giving more exotic examples, we begin by summarizing the structure of cyclic-equivariant $K$-theory.

\subsubsection{Complex circle-equivariant \texorpdfstring{$K$}{K}-theory}

Let $T \subset \bbC^\times$ be the circle group. We begin by describing $KU_T$. Let $L$ be the tautological complex character of $T$, so that
\[
\pi_0 KU_T \cong RU(T) \cong \bbZ[L^{\pm 1}].
\]
If $\alpha \in RO(T)$, then either $\alpha$ or $\alpha+1$ lifts to $RU(T)$, but this lift is not canonical. There are several constructions that depend on the choice of a complex structure, and for this reason it can be convenient to think of $T$-equivariant homotopy groups as graded not over $RO(T)$, but over ``$RU(T)$ adjoined $\bbR = \tfrac{1}{2}\bbC$'', i.e.\ the group
\[
\left(\bbZ\{L^n : n \in\bbZ\} \oplus \bbZ\{1\}\right)/(L^0 - 2).
\]

Associated to every virtual complex representation $\alpha = \sum_i n_i L^i$ is the invertible Bott class
\[
\beta_\alpha = \prod_i \beta_{L^i}^{n_i} \in \pi_\alpha KU_T\cong RU(T)\{\beta_\alpha\}.
\]
Together with $\pi_1 KU_T = 0$, this completely determines $\pi_\star KU_T$.

The Adams operation $\psi^{-1}$ acts on $\pi_\star KU_T$ by multiplicative automorphisms, satisfying
\[
\psi^{-1}(L) = L^{-1},\qquad \psi^{-1}(\beta_{L^i}) = -L^{-i}\cdot \beta_{L^i}.
\]
The Adams operation $\psi^k$ for $k>0$ acts on $\pi_V KU_T\cong \widetilde{KU}{}_0^T(S^V)$ for $V$ an actual representation by ring endomorphisms, where it is determined by
\[
\psi^k(L) = L^k,\qquad \psi^k(\beta_{L^i}) = (1+L^i+L^{2i}+\cdots+L^{(k-1)i})\cdot \beta_{L^i}.
\]
See for example \cite{atiyahtall1969group}. Finally, if $V = \sum_i n_i L^i$ with $n_i\geq 0$, then the Euler class $e_V\in RU(T)$ is given by
\[
e_V = \prod_i (1-L^i)^{n_i}.
\]

\subsubsection{Real circle-equivariant \texorpdfstring{$K$}{K}-theory}

We now descend to $KO_T$. A virtual $T$-representation $\alpha = \sum_i n_iL^i$ admits a Spin structure if and only if its second Stiefel--Whitney class
\[
w_2(\alpha) \equiv \sum_i n_i \cdot i \pmod{2}
\]
vanishes. In this case $\alpha$ is $KO$-orientable, in the sense that there is an equivalence
\[
F(S^\alpha,KO_T)\simeq F(S^{|\alpha|},KO_T).
\]
If $|\alpha|$ is a multiple of $8$, then this is realized by an invertible Bott class 
\[
\beta_\alpha^{\Spin} \in \pi_{\alpha}KO_T.
\]
However, if $\psi^{-1}(\alpha) \neq \alpha$, then the image of $\beta_\alpha^{\Spin}$ under the complexification $c\colon KO_T\rightarrow KU_T$ is \textit{not} guaranteed to be the complex Bott class $\beta_\alpha \in \pi_\alpha KU_T$. Instead, the value of $c(\beta_\alpha^{\Spin})$ can be determined as follows. Abbreviate $n = \sum_i  n_i\cdot i$. Under the assumption that $|\alpha|$ is a multiple of $8$, we then have
\[
\psi^{-1}(\beta_\alpha) = L^{-n}\cdot \beta_\alpha.
\]
It follows that $\pm L^{-n/2}\beta_\alpha$ are the only units in $\pi_\alpha KU_T$ fixed by $\psi^{-1}$. As the Bott class is compatible with restriction, we must have $\res^T_e(c(\beta_\alpha^{\Spin})) = \beta^{|\alpha|/2}$, and the only possibility is that
\[
c(\beta_\alpha^{\Spin}) = L^{-n/2}\beta_\alpha.
\]

We can now describe $\pi_\star KO_T$ in general. For any $\alpha = \sum_i n_i L^i$, either $\alpha$ or $\alpha+L$ is $\Spin$, so the above discussion provides a $KO_T$-linear equivalence between $\Sigma^{-\alpha}KO_T$ and an integer suspension of one of $KO_T$ or $\Sigma^{-L}KO_T$. Hence it suffices to describe the fixed points of these.

In the former case we have $\pi_0 KO_T = RO(T)$ and in general
\[
KO^T \cong KO\{1\}\oplus\bigoplus_{i>0}KU,
\]
where
\[
\im(\pi_{2n}KO_T\rightarrow\pi_{2n}KU_T) = \begin{cases}\left(\bbZ\{1\}\oplus\bbZ\{L^i+L^{-i}:i\geq 1\}\right)\beta^n, & n\equiv 0\pmod{4},\\
\left(\bbZ\{2\}\oplus\bbZ\{L^i+L^{-i}:i\geq 1\}\right)\beta^n,&n\equiv 2\pmod{4},\\
\bbZ\{L^i-L^{-i}:i\geq 1\}\beta^n,&n\equiv 1,3\pmod{4}.
\end{cases}
\]
In the latter case, $H^{>0}(C_2;\pi_{\ast+L}KU_T) = 0$, and it follows
\[
(\Sigma^{-L}KO_T)^T\simeq \bigoplus_{i\geq 0}KU,
\]
where
\[
\im(\pi_{2n+L}KO_T\rightarrow \pi_{2n+L}KU_T) = \begin{cases}
\bbZ\{L^i-L^{-i-1}\}\beta^n\beta_L,&n\equiv 0\pmod{2},\\
\bbZ\{L^i+L^{-i-1}\}\beta^n\beta_L,&n\equiv 1\pmod{2}.
\end{cases}
\]
This completely determines $\pi_\star KO_T$.

\subsubsection{Finite cyclic groups}

Let $C_n\subset T$ denote the finite subgroup of order $n$, so
\[
\pi_0 KU_{C_n}\cong RU(C_n)\cong \bbZ[L]/(L^n-1).
\]
If $\alpha = \sum_i n_i L^i$, then by restricting the above $KO_T$-linear equivalences we obtain $KO_{C_n}$-linear equivalences between $\Sigma^\alpha KO_{C_n}$ and an integer suspension of either $KO_{C_n}$ or $\Sigma^{-L}KO_{C_n}$. For $n$ even, write $\sigma$ for the real sign representation of $C_n$, satisfying $\bbC \otimes \sigma = L^{n/2}$. Then we can identify
\[
KO^{C_n} \simeq \begin{cases}
KO\{1,\sigma\} \oplus KU^{n/2-1},&n\equiv 0 \pmod{2},\\
KO\{1\}\oplus KU^{(n-1)/1},&n\equiv 1 \pmod{2},
\end{cases}
\]
and
\[
(\Sigma^{2-L}KO_{C_n})^{C_n}\simeq \begin{cases}
KU^{n/2},&n\equiv 0 \pmod{2},\\
KO\{L^{(n-1)/2}\beta_L\beta^{-1}\}\oplus KU^{(n-1)/2},&n\equiv 1 \pmod{2},
\end{cases}
\]
with the images of $\pi_\ast KO_{C_n}$ and $\pi_{\ast+L}KO_{C_n}$ in $\pi_\star KU_{C_n}$ easily determined as in the $T$-equivariant case. This completely describes $\pi_\star KO_{C_n}$ if $n$ is odd, but if $n$ is even then one must also account for degrees involving $\sigma$. For this reason it can be convenient to grade $C_n$-equivariant computations over ``$RU(C_n)$ adjoined $\bbR = \frac{1}{2}\bbC$ and $\sigma = \frac{1}{2}L^{n/2}$'', i.e.\
\[
\left(\bbZ\{L^0,\ldots,L^{n-1}\}\oplus \bbZ\{1,\sigma\}\right)/(L^0-2,~L^{n/2}-\sigma),
\]
in order to incorporate a choice of complex structure into the grading. If $\alpha$ is an element therein, then the above discussion gives a $KO_{C_n}$-linear equivalence between $\Sigma^\alpha KO_{C_n}$ and an integer suspension of one of
\[
KO_{C_n},\quad \Sigma^{-L} KO_{C_n},\quad \Sigma^{-\sigma}KO_{C_n},\quad \Sigma^{-L-\sigma}KO_{C_n}.
\]
The first two cases were described above, and the latter two can be computed using the cofiber sequence
\[
S^{-\sigma}\rightarrow S^0\rightarrow D(C_n/C_{n/2+}),
\]
ultimately allowing one to identify
\[
(\Sigma^{-\sigma}KO_{C_n})^{C_n} \simeq \begin{cases} KO\oplus \Sigma^{-1}KO\oplus KU^{n/4},&n\equiv 0\pmod{4},\\
KO\{1\}\oplus KU^{(n-2)/4},&n\equiv 2 \pmod{4},
\end{cases}
\]
as an augmentation ideal of $KO^{C_n}$, and similarly
\[
(\Sigma^{2-L-\sigma}KO_{C_n})^{C_n} \simeq \begin{cases}
KU^{n/4},&n\equiv 0 \pmod{4},\\
KU^{(n-2)/4}\oplus \Sigma^{-2}KO,&n\equiv 2 \pmod{4}.
\end{cases}
\]

\begin{ex}
Let us give details for the identification of $(\Sigma^{2-L-\sigma}KO_{C_n})^{C_n}$ for $n\equiv 2 \pmod{4}$, as the rest are similar or easier. The fiber sequence
\[
(\Sigma^{2-L-\sigma}KO_{C_n})^{C_n}\rightarrow (\Sigma^{2-L}KO_{C_n})^{C_n}\xrightarrow{\res} (\Sigma^{2-L}KO_{C_{n/2}})^{C_{n/2}},
\]
may be identified as
\begin{align*}
(\Sigma^{2-L-\sigma}&KO_{C_n})^{C_n}\rightarrow KU\{(L^i-L^{-i-1})\beta_L\beta^{-1} : 0 \leq i < \tfrac{n}{2}\}\\
&\xrightarrow{\res}KO\{L^{(n-2)/4}\beta^{-1}\beta_L\} \oplus KU\{(L^i-L^{-i-1})\beta_L\beta^{-1} : 0 \leq i < \tfrac{n-2}{4}\},
\end{align*}
where we name summands for how generators in $\pi_0$ are named after base change to $KU$. We always have
\[
\res^{C_n}_{C_{n/2}}((L^i+L^{-i-1})\beta_L\beta^{-1}) = \res^{C_n}_{C_{n/2}}((L^{i+n/2}+L^{-i-n/2-1})\beta_L\beta^{-1}).
\]
If $0\leq i <(n-2)/4$ then these terms are distinct before restriction, allowing us to split off a copy of
\[
KU\{((L^i+L^{-i-1})-(L^{i+n/2}+L^{-i-n/2-1}))\beta_L\beta^{-1} : 0 \leq i < \tfrac{n-2}{4}\}
\]
in the fiber. Thus $(\Sigma^{2-L-\sigma}KO_{C_n})^{C_n}\simeq KU^{(n-2)/4}\oplus F$ where
\[
F = \Fib\left(r\colon KU\{ (L^{(n-2)/4}+L^{(-n-2)/4})\beta^{-1}\beta_L\}\rightarrow KO\{L^{(n-2)/4}\beta^{-1}\beta_L\}\right),
\]
which we are claiming is equivalent to $\Sigma^{-2}KO$. Observe that $r$ sends the generator of $\pi_4 KU$ to that of $\pi_4 KO$. The Wood cofiber sequence $\Sigma KO\xrightarrow{\eta}KO\rightarrow KU$ yields
\[
\Mod_{KO}(KU,KO)\simeq KU,
\]
so this in fact characterizes $r$. Thus $r$ is equivalent to the twofold desuspension of the boundary map in the Wood cofiber sequence, implying $F \simeq \Sigma^{-2}KO$ as claimed.
\tqed
\end{ex}

\subsubsection{Examples}

There is already an extensive literature on the $K$-theory and $J$-theory of the lens spaces $S(kL)/C_n$, more than we can hope to summarize here. We just give some small examples illustrating general phenomena relevant to our study of periodicities.

\begin{ex}
Take $G = C_2$. The behavior of the $u_\sigma = t_{1-\sigma}$-elements in $\pi_{2^{\gamma(m)}(1-\sigma)}C(a_\sigma^{m+1})$ discussed in the introduction was analyzed in detail by Araki and Iriye in \cite[Section 3]{arakiiriye1982equivariant}. We give an example to highlight the nature of compatibility between these elements as $m$ varies.

As $4\sigma = \bbH\otimes_\bbR \sigma$ is $4$-dimensional quaternionic, we have
\[
\pi_{4\sigma}KO_{C_2} \cong RO(C_2)\cdot 2 \beta_{4\sigma},\qquad \pi_{8\sigma}KO_{C_2}\cong RO(C_2)\cdot \beta_{8\sigma},
\]
where
\[
a_\sigma^4\cdot \beta_{8\sigma} = (1-\sigma)\cdot 2 \beta_{4\sigma},\quad a_\sigma^4\cdot 2\beta_{4\sigma} = 4(1-\sigma),\quad a_\sigma^8\cdot \beta_{8\sigma} = 16(1-\sigma).
\]
Hence there exist $u_\sigma$-elements
\[
u_{4\sigma} = J(2\beta_{4\sigma}) \in \pi_{4(1-\sigma)}C(a_\sigma^4),\qquad u_{8\sigma} = J(\beta_{8\sigma}) \in \pi_{8(1-\sigma)}C(a_\sigma^8),
\]
and up to choices of orientations these satisfy
\[
\res^{C_2}_e(\partial(u_{4\sigma})) = \nu \in \pi_3 S,\qquad \res^{C_2}_e(\partial(u_{8\sigma})) = \sigma \in \pi_7 S.
\]
As $\res^{C_2}_e(\partial(u_{4\sigma}^2)) = 2\nu \neq 0$, it follows $u_{4\sigma}^2$ cannot lift to $C(a_\sigma^5)$, so the map $q\colon C(a_\sigma^8)\rightarrow C(a_\sigma^4)$ must satisfy $q(u_{8\sigma})\neq u_{4\sigma}^2$. The difference $\epsilon = u_{4\sigma}^2 \cdot q(u_{8\sigma})^{-1} $ is measured by
\begin{align*}
\epsilon &= u_{4\sigma}^{2}\cdot q(u_{8\sigma})^{-1} = J(2\beta_{4\sigma})^2 \cdot J(a_\sigma^4 \beta_{8\sigma})^{-1}\\
&= J(2\cdot 2\beta_{4\sigma} - a_\sigma^4 \beta_{8\sigma}) = J((1+\sigma)\cdot 2\beta_{4\sigma}).
\end{align*}\
Thus $\epsilon \in \pi_0 C(a_\sigma^4)^\times$ is a class satisfying $\res^{C_2}_e(\partial(\epsilon)) = 2\nu$. In fact
\[
J((1+\sigma)\cdot 2\beta_{4\sigma}) = J(\tr_e^{C_2}(2\beta^2)) = N_e^{C_2}(J(2\beta^2))
\]
where $N_e^{C_2}\colon \pi_0^e C(a_\sigma^4)\rightarrow \pi_0^{C_2}C(a_\sigma^4)$ is the norm, and where if we write $\pi_0^e C(a_\sigma^4) \cong \pi_0 D(S^3_+)\cong \bbZ[\ol{\nu}]/(\ol{\nu}^2,24\ol{\nu})$ then $J(2\beta^2) = 1+\ol{\nu}$.
\tqed
\end{ex}

\begin{ex}
We give an example of the remarks at the end of \cref{ssec:htpyequivreps}. We were guided to existence of an example like this by \cite{kobayashi1977stable}. Take $G = C_8$ and let
\[
\rho = L + L^3 + L^5 + L^7.
\]
This representation is quaternionic of real dimension $8$, and thus the Bott class $\beta_\rho \in \pi_\rho KU_{C_8}$ descends to $KO_{C_8}$, yielding
\[
\pi_\rho KO_{C_8}\cong RO_{C_8}\{\beta_\rho\}.
\]
Set $RO_{C_8} = \bbZ\{1,\sigma,\lambda,\mu,\mu'\}$ where
\[
\bbC\otimes\sigma = L^4,\quad \bbC\otimes\lambda = L^2+L^6,\quad \bbC\otimes \mu = L+L^7,\quad \bbC\otimes\mu' = L^3+L^5.
\]
Then $\psi^3(\mu) = \mu'$ implies $S^{\mu}[\tfrac{1}{3}]\simeq S^{\mu'}[\tfrac{1}{3}]$, and $(\psi^3-\psi^1)^2(\mu) = 2(\mu-\mu')$ implies $S^{2\mu}\simeq S^{2\mu'}$. If we identify $\mu$ and $\mu'$ as the underlying real representations of $L^7$ and $L^5$ respectively, then a particular choice of invertible element $\psi_3 \in \pi_{\mu'-\mu}S_{C_8}[\tfrac{1}{3}]$ is constructed in \cref{ex:powermap}.

The Euler class of $\rho$ is given by
\[
e_\rho = (1-L)(1-L^3)(1-L^5)(1-L^7) = 4+2\sigma+2\lambda-2\mu-2\mu'.
\]
A calculation reveals that $2-\mu$ has order $16$ in $RO(C_8)/(e_\rho)$, with
\[
16(2-\mu) = e_\rho\cdot (11-5\sigma-\lambda-4\mu+4\mu');
\]
so the best that the $J$-homomorphism gives is an invertible element
\[
u_{16\mu} = J((11-5\sigma-\lambda-4\mu+4\mu')\beta_\rho) \in \pi_{16(2-\mu)}^{C_8}C(a_\rho).
\]
As $\res^{C_8}_e((11-5\sigma-\lambda-4\mu+4\mu')\beta_\rho) = 4\beta^4 \in \pi_8 KO$, this element satisfies $\res^{C_8}_e(\partial(u_{16\mu})) = \pm 4 \sigma$, and in the context of \cref{thm:hfpss} this yields a differential 
\[
d_8(u_\mu^{16}) = \pm 4\sigma\cdot a_\rho u_\rho^{-1} \cdot u_\mu^{16}
\]
in the $C_8$-homotopy fixed point spectral sequence. Here, $a_\rho u_\rho^{-1}$ generates $H^8(C_8;\pi_0 S)$.

On the other hand, $8\mu$ is locally $J$-equivalent to $3\mu+5\mu'$, and another calculation shows
\[
16-(3\mu+5\mu') = e_\rho \cdot (5-3\sigma-\lambda-\mu'),
\]
thus giving $u_{3\mu+5\mu'} \in \pi_{16-3\mu-5\mu'}^{C_8} C(a_\rho)$. Hence we obtain an invertible element
\[
\psi_3^5\cdot u_{3\mu+5\mu'} \in \pi_{8(2-\mu)} ^{C_8}C(a_\rho)[\tfrac{1}{3}]
\]
playing the role of ``$u_{8\mu}$''. In the context of \cref{thm:hfpss}, by \cref{ex:jdiff} this yields a differential
\[
d_8(u_\mu^8) = \pm 2 \sigma\cdot a_\rho u_\rho^{-1}\cdot u_\mu^8
\]
in the $C_8$-homotopy fixed point spectral sequence, refining the above differential on $u_\mu^{16}$.
\tqed
\end{ex}

\begin{ex}
Let $G = C_6$ and $\rho = L+L^5$. Then $\alpha = (1-L^2-L^3+L^5)$ satisfies 
\[
e_\rho \cdot\alpha = \alpha,
\]
so there exist compatible $t_\alpha$-elements in $\pi_\alpha^{C_6} C(a_\rho^n)$ for all $n\geq 1$. As $S(\infty\rho)$ is a model for $EC_6$, by taking $n\rightarrow\infty$ this produces a $t_\alpha$-element in the completion $F(EC_{6+},S_{C_6})$. This phenomenon is generic for composite order groups, corresponding to the kernel of the completion map $RU(G)\rightarrow KU^0BG$.
\tqed
\end{ex}

\subsection{Example: the symmetric group on \texorpdfstring{$3$}{3} letters}

We give an example with an orthogonal irreducible. Let $G = \Sigma_3$ be the symmetric group on $3$ letters. Write $\sigma$ for its real sign representation and $\lambda$ for its reduced real canonical permutation representation, so that
\[
RO(\Sigma_3)\cong\bbZ\{1,\sigma,\lambda\},\qquad \sigma^2 = 1,\qquad \sigma\lambda = \lambda,\qquad \lambda^2 = 1+\sigma+\lambda.
\]
Complexification $RO(\Sigma_3)\rightarrow RU(\Sigma_3)$ is an isomorphism, and we write $RU(\Sigma_3) = \bbZ\{1,\sigma_\bbC,\lambda_\bbC\}$. Consider the sequence
\begin{center}\begin{tikzcd}
RU(\Sigma_3)\cong\pi_0 KU_{\Sigma_3}&\ar[l,"a_\lambda"']\pi_\lambda KU_{\Sigma_3}& \pi_{2\lambda}KU_{\Sigma_3}\cong RU(\Sigma_3)\{\beta_{\lambda_\bbC}\} \ar[l,"a_\lambda"']
\end{tikzcd}.\end{center}
The composite is determined by
\[
a_\lambda^2 \beta_{\lambda_\bbC} = a_{\lambda_\bbC} \beta_{\lambda_\bbC} =  e_{\lambda_\bbC} = 1 - \lambda_\bbC + \sigma_\bbC,
\]
and has image the rank $1$ subspace $\bbZ\{1-\lambda_\bbC+\sigma_\bbC\}\subset RU(\Sigma_3)$. As $\pi_\lambda KU_{\Sigma_3}$ is a free abelian group of rank $1$ \cite{karoubi2002equivariant}, the only possibility is that $\pi_\lambda KU_{\Sigma_3} \cong \bbZ\{a_\lambda\beta_{\lambda_\bbC}\}$ with $\sigma_\bbC a_\lambda = a_\lambda$ and $\lambda_\bbC a_\lambda = -a_\lambda$.
Thus $\pi_{\ast\lambda}KU_{\Sigma_3}$ has the $2$-periodic pattern
\begin{center}\begin{tikzcd}[ampersand replacement=\&, column sep=1.6cm]
\bbZ\{1,\sigma_\bbC,\lambda_\bbC\}\&\ar[l,"a_\lambda"']\bbZ\{a_\lambda\beta_{\lambda_\bbC}\}\&\ar[l,"a_\lambda"']\bbZ\{1,\sigma_\bbC,\lambda_\bbC\}\beta_{\lambda_\bbC}\&\ar[l,"a_\lambda"']
\end{tikzcd}.\end{center}
Complex conjugation acts trivially on $RU(\Sigma_3)$, and using $\res^{\Sigma_3}_{C_2}(\lambda) \cong 1+\sigma$ nd $\res^{\Sigma_3}_{C_3}(\lambda_\bbC) \cong L+L^{-1}$ we find
\[
\psi^{-1}(\beta_{\lambda_\bbC}) = \sigma_\bbC \beta_{\lambda_{\bbC}}.
\]
Thus $H^0(C_2;\pi_{\ast\lambda}KU_{\Sigma_3})$ is $4$-periodic with
\[
H^0(C_2;\pi_{m\lambda}KU_{\Sigma_3}) = \begin{cases} \bbZ\{1,\sigma_\bbC,\lambda_\bbC\},&m=0,\\
\bbZ\{a_\lambda\beta_{\lambda_\bbC}\},&m=1,\\
\bbZ\{1+\sigma_\bbC,\lambda_\bbC\}\beta_{\lambda_\bbC},&m=2,\\
\bbZ\{a_\lambda\beta_{\lambda_\bbC}^2\},&m=3,
\end{cases}
\]

We claim that $H^0(C_2;\pi_{\ast\lambda}KU_{\Sigma_3})$ consists of permanent cycles. As $4\lambda = \lambda\otimes_\bbR\bbH$ is $8$-dimensional quaternionic, we reduce to considering $H^0(C_2;\pi_{m\lambda}KU_{\Sigma_3})$ for $m\in\{0,1,2,3\}$. As $\eta\cdot (1+\sigma_\bbC)\beta_{\lambda_\bbC} = 0$, necessarily $(1+\sigma_\bbC)\beta_{\lambda_\bbC}$ is a permanent cycle. As $a_\lambda$ is a permanent cycle and
\[
\langle a_\lambda,a_\lambda\beta_{\lambda_\bbC}\rangle = 1,\qquad \langle a_\lambda^2,\lambda_\bbC \beta_{\lambda_\bbC}\rangle = -1,\qquad \langle a_\lambda^3,a_\lambda\beta_{\lambda_\bbC}^2\rangle = 3,
\]
we deduce following \cref{rmk:dualityhfpss} that the remaining classes are permanent cycles.  Thus $\pi_{\ast\lambda}KO_{\Sigma_3}$ has the $4$-periodic pattern
\begin{center}\begin{tikzcd}[row sep=tiny]
\bbZ\{1,\sigma,\lambda\}&\ar[l,"a_\lambda"']\bbZ\{a_\lambda\beta_{\lambda_\bbC}\}&\ar[l,"a_\lambda"']\bbZ\{(1+\sigma),\lambda\}\beta_{\lambda_\bbC}\\
\hphantom{\bbZ\{1,\sigma,\lambda\}}&\ar[l,"a_\lambda"']\bbZ\{a_\lambda\beta_{\lambda_\bbC}^2\}&\ar[l,"a_\lambda"']\bbZ\{1,\sigma,\lambda\}\beta_{\lambda_\bbC}^2&\ar[l,"a_\lambda"']{}
\end{tikzcd},\end{center}
where we now write $\sigma$ and $\lambda$ for classes that complexify to $\sigma_\bbC$ and $\lambda_\bbC$. In particular, the identities
\[
3^{2m}(1+\sigma-\lambda) = a^{4m+2}_\lambda\cdot -\lambda\beta_{\lambda_\bbC}^{2m+1},\qquad 3^{2m-1}(1+\sigma-\lambda) = a^{4m}_\lambda\cdot \beta_{\lambda\bbC}^{2m}
\]
imply that $C(a_\lambda^{m+1})$ admits a real $t_{1+\sigma-\lambda}$-element of order $3^{\lfloor m/2\rfloor}$. Moreover, the real $t_{1+\sigma-\lambda}$-elements
\[
t_{3^{n-1}(1+\sigma-\lambda)} = \begin{cases} J(\beta_{\lambda_\bbC}^{n}),&k\text{ even},\\
J(-\lambda\beta_{\lambda_\bbC}^{n}),&k\text{ odd},
\end{cases} \quad \in \quad \pi_{3^{n-1}(1+\sigma-\lambda)}^{\Sigma_3}C(a_\lambda^{2m})
\]
give rise to infinite periodic families
\[
\partial(t_{3^{n-1}(1+\sigma-\lambda)}^k) \in \pi_{3^{n-1}k(1+\sigma-\lambda)+2n\lambda-1}S_{\Sigma_3}
\]
with the property that
\[
\res^{\Sigma_3}_e(\partial(t_{3^{n-1}(1+\sigma-\lambda)}^k)) = k\cdot j_{4n-1} \in \pi_{4n-1}S.
\]

\subsection{The dihedral group of order \texorpdfstring{$8$}{8}}

We give a more delicate example with an orthogonal irreducible. Let $D_8$ be the dihedral group of order $8$, generated by a rotation $r$ and reflection $f$.  We then have
\[
RO(D_8) = \bbZ\{1,\sigma_r,\sigma_f,\sigma_{rf},\rho\},
\]
where $\sigma_g$ is the character with kernel $\langle r^2,g\rangle$ and $\rho$ is the tautological $2$-dimensional representation. These satisfy the usual identity between characters together with
\[
\sigma_g\cdot \rho = \rho,\qquad \rho^2 = 1+\sigma_r+\sigma_f+\sigma_{rf}.
\]

We begin by describing $\pi_{\ast\rho}KU_{D_8}$. Complexification $RO(D_8)\rightarrow RU(D_8)$ is an isomorphism. We will use this to write $RU(D_8) = \bbZ\{1,\sigma_r,\sigma_f,\sigma_{rf},\rho\}$, only with the understanding that in the context of representation-grading these symbols refer only to their real counterparts. Write $\beta_{2\rho}\in \pi_{2\rho}KU_{D_8} = \pi_{\bbC\otimes\rho}KU_{D_8}$ for the Bott class of the complex representation $\bbC\otimes\rho$ of real dimension $4$, and let $e_{\rho_\bbC} = a_\rho^2 \beta_{2\rho} \in RU(D_8)$. As $r$ acts by orientation-preserving automorphisms of $\rho$ but $f$ and $rf$ do not, we find $\Lambda^2\rho = \sigma_r$, so that
\[
e_{\rho_\bbC} = 1 + \sigma_r - \rho.
\]
This determines $\pi_{2\ast\rho}KU_{D_8}$. We next compute $\pi_\rho KU_{D_8}$. A computation shows
\begin{align*}
\ker(e_{\rho_\bbC}) &= \bbZ\{1-\sigma_r,\sigma_f-\sigma_{rf},1+\sigma_f+\rho\},\\
(e_{\rho_\bbC}) &= \bbZ\{1+\sigma_r-\rho,\sigma_f+\sigma_{rf}-\rho\}.
\end{align*}
In the terminology of \cite{karoubi2002equivariant}, only the conjugacy classes of $e$, $r$, and $r^2$ are oriented and even with respect to $\rho$, and thus $\pi_\rho KU_{D_8} = \bbZ^3$. As $\ker(a_\rho\colon \pi_{2\rho}KU_{D_8}\rightarrow\pi_\rho KU_{D_8})$ has rank $3$ and $\im(a_\rho^2\colon \pi_{2\rho}KU_{D_8}\rightarrow\pi_0 KU_{D_8}) = (e_{\rho_\bbC})\subset RU(D_8)$ is a split summand, the only possibility is that
\[
\pi_\rho KU_{D_8} = \bbZ\{a_\rho\beta_{2\rho},a_\rho \sigma_f \beta_{2\rho},b\}
\]
for some $b$ generating the kernel of $a_{\rho}$. This determines $\pi_{\ast\rho}KU_{D_8}$.

We next descend to $\pi_{\ast\rho}KO_{D_8}$. As $4\rho = \bbH \otimes_\bbR \rho$ is $8$-dimensional quaternionic, we reduce to considering $0\leq \ast \leq 3$. As
\[
\res^{D_8}_{\langle r \rangle}(\rho) = L+L^{-1},\quad \res^{D_8}_{\langle f\rangle}(\rho) = 1+\sigma = \res^{D_8}_{\langle rf\rangle}(\rho),
\]
we must have
\[
\psi^{-1}(\beta_{2\rho}) = \sigma_r \beta_{2\rho}.
\]
As $b$ must lift a multiple of $\beta$ and generates the kernel of $a_\rho$ we also have $\psi^{-1}(b) = -b$. Thus 
\[
H^0(C_2;\pi_{m\rho}KU_{D_8}) = \begin{cases}
\bbZ\{1,\sigma_r,\sigma_f,\sigma_{rf},\rho\},&m=0,\\
\bbZ\{a_\rho,a_\rho\sigma_f\}\beta_{2\rho},&m=1,\\
\bbZ\{(1+\sigma_r),(\sigma_f+\sigma_{rf}),\rho\}\beta_{2\rho},&m=2,\\
\bbZ\{a_\rho,a_\rho\sigma_f\}\beta_{2\rho}^2,&m=3.
\end{cases}
\]
We claim that all of these are permanent cycles. For $m=0$ this holds as $RO(D_8) \cong RU(D_8)$, and using $a_\rho$ this implies it for $m=1$. Necessarily $(1+\sigma_r)\beta_{2\rho}$ and $(\sigma_f+\sigma_{rf})\beta_{2\rho}$ are permanent cycles as they annihilate $\eta$, and $a_\rho\beta_{2\rho}^2$ and $a_\rho\sigma_f\beta_{2\rho}^2$ are permanent cycles because the same is true of each of $a_\rho$, $\beta_{2\rho}^2$, and $\sigma_f$. Finally, as $a_\rho \cdot \rho\beta_{2\rho} = a_\rho(1+\sigma_f)\beta_{2\rho}$, we find that $a_\rho\colon H^{3}(C_2;\pi_{2\rho+2}KU_{D_8})\rightarrow H^{3}(C_2;\pi_{\rho+2}KU_{D_8})$ is an injection, forcing $\rho \beta_{2\rho}$ to be a permanent cycle. In the end we find that $\pi_{\ast\rho}KO_{D_8}$ (modulo possible torsion classes) has the $4$-periodic pattern
\begin{center}\begin{tikzcd}[row sep=tiny]
\bbZ\{1,\sigma_r,\sigma_f,\sigma_{rf},\rho\}&\ar[l,"a_\rho"']\bbZ\{a_\rho,a_\rho\sigma_f\}\beta_{2\rho}&\ar[l,"a_\rho"']\bbZ\{(1+\sigma_r),(\sigma_f+\sigma_{rf}),\rho\}\beta_{2\rho}\\
\hphantom{\bbZ\{1,\sigma_r,\sigma_f,\sigma_{rf},\rho\}}&\ar[l,"a_\rho"']\bbZ\{a_\rho,a_\rho\sigma_f\}\beta_{2\rho}^2&\ar[l,"a_\rho"']\bbZ\{1,\sigma_r,\sigma_f,\sigma_{rf},\rho\}\beta_{2\rho}^2
\end{tikzcd}.\end{center}

Let $\alpha = 1+\sigma_r-\rho$. Then the identities
\begin{align*}
2^{4m-1}\alpha &= a_\rho^{4m}((1+2^{2m-1})+(1-2^{2m-1})\sigma_f)\beta_{2\rho}^{2m}\\
2^{4m+1}\alpha &= a_\rho^{4m+2}(2^{2m-1}(1+\sigma_r)-2^{2m-1}(\sigma_f+\sigma_{rf})-\rho)\beta_{2\rho}^{2m+1}
\end{align*}
for $m>0$ produce $t_\alpha$-elements
\[
t_{2^{2n-1}\alpha} \in \pi_{2^{2n-1}\alpha}^{D_8}C(a_\rho^{2n})
\]
satisfying
\[
\res^{D_8}_e(\partial(t_{2^{2n-1}\alpha})) = \begin{cases} 2 j_{4n-1} ,& n \equiv 0 \pmod{2},\\
j_{4n-1}, &n\equiv 1 \pmod{2}.
\end{cases}
\]
Likewise, because
\[
\res^{D_8}_{\langle r^2\rangle}(a_\rho \beta_{2\rho}) = (a_\sigma \beta_{2\sigma})^2 = \eta_{C_2}^2,
\]
the element
\[
t_\alpha = J(a_\rho \beta_{2\rho}) \in \pi_\alpha^{D_8} C(a_\rho)
\]
satisfies
\[
\res^{D_8}_{\langle r^2\rangle}(t_\alpha) = u_{2\sigma} \in \pi_{2(1-\sigma)}^{C_2}C(a_\sigma^2).
\]

\subsection{Example: the nonabelian group of order \texorpdfstring{$21$}{21}}

We give an example with a complex irreducible. Let $G = C_7\rtimes C_3 = \langle x,y:x^7=e=y^3,xy=yx^2\rangle$ be the nonabelian group of order $21$. Consulting \cite{dokchitserxxxxgroupnames} we see that this group has the five conjugacy classes
\begin{gather*}
C(e) = \{e\},\qquad C(x) = \{x,x^2,x^4\},\qquad C(x^3) = \{x^3,x^5,x^6\},\\
C(y) = \{y,yx,\ldots,yx^6\},\qquad C(y^2) = \{y^2,y^2x,\ldots,y^2x^6\},
\end{gather*}
and character table
\[
\begin{array}{c|ccccc}
&C(e)&C(x)&C(x^3)&C(y)&C(y^2)\\
\hline
1&1&1&1&1&1\\
\omega&1&1&1&\zeta_3^2&\zeta_3\\
\ol{\omega}&1&1&1&\zeta_3&\zeta_3^2\\
\rho&3&\zeta_7^3+\zeta_7^6+\zeta_7^5&\zeta_7+\zeta_7^2+\zeta_7^4&0&0\\
\ol{\rho}&3&\zeta_7+\zeta_7^2+\zeta_7^4&\zeta_7^3+\zeta_7^6+\zeta_7^5&0&0\\
\end{array}.
\]
Write $\omega_0$ and $\rho_0$ for the underlying real representations of $\omega$ and $\rho$. We describe $\pi_{\ast\rho_0}KO_{C_7\rtimes C_3}$. Observe
\[
\omega\rho = \rho = \ol{\omega}\rho,\qquad \rho^2 = \rho + 2 \ol{\rho},\qquad \rho\ol{\rho} = 1+\omega+\ol{\omega}+\rho+\ol{\rho}.
\]
Using the general character identity 
\[
\chi_{\Lambda^2U}(g) = \frac{\chi_U(g)^2 - \chi_U(g^2)}{2},
\]
we compute $\Lambda^2 \rho = \ol{\rho}$ and thus
$
e_\rho = \ol{\rho}-\rho
$
This determines $\pi_{\ast\rho}KU_{C_7\rtimes C_3}$. Restriction gives an injection $\pi_\rho KU_{C_7\rtimes C_3}\rightarrow \pi_{L^3+L^5+L^6}KU_{C_7}\times \pi_{L^0+L^1+L^2} KU_{C_3}$. Using this we compute $\psi^{-1}(\beta_\rho) = - \beta_\rho$. The same calculations works swapping $\rho$ and $\ol{\rho}$, and so we find
\[
H^0(C_2;\pi_{m_1\rho+m_2\ol{\rho}}KU_{C_7\rtimes C_3}) = \begin{cases}\bbZ\{1,\omega+\ol{\omega},\rho+\ol{\rho}\}\beta_\rho^{m_1}\beta_{\ol{\rho}}^{m_2},&m_1+m_2\text{ even},\\
\bbZ\{\omega-\ol{\omega},\rho-\ol{\rho}\}\beta_\rho^{m_1}\beta_{\ol{\rho}}^{m_2},&m_1+m_2\text{ odd}.
\end{cases}
\]
As $\bbH\otimes_\bbC \rho = \rho + \ol{\rho}$ is $12$-dimensional quaternionic, $\beta_\rho^2\beta_{\ol{\rho}}^2$ is a permanent cycle but $\beta_\rho\beta_{\ol{\rho}}$ is not. Thus we can compute $\pi_{\ast\rho_0}KO_{C_7\rtimes C_3}$, with a convenient choice of complex structure on multiples of $\rho_0$, as having the form
\begin{center}\begin{tikzcd}[row sep=0mm, column sep=small]
\bbZ\{1,\omega+\ol{\omega},\rho+\ol{\rho}\}&\ar[l,"a_\rho"']\bbZ\{\omega-\ol{\omega},\rho-\ol{\rho}\}\beta_\rho&\ar[l,"a_{\ol{\rho}}"']\bbZ\{2,\omega+\ol{\omega},\rho+\ol{\rho}\}\beta_\rho\beta_{\ol{\rho}}\\
\hphantom{\bbZ\{1,\omega+\ol{\omega},\rho+\ol{\rho}\}}&\ar[l,"a_\rho"']\bbZ\{\omega-\ol{\omega},\rho-\ol{\rho}\}\beta_\rho^2\beta_{\ol{\rho}}&\ar[l,"a_{\ol{\rho}}"']\bbZ\{1,\omega+\ol{\omega},\rho+\ol{\rho}\}\beta_\rho^2\beta_{\ol{\rho}}^2&\ar[l,"a_\rho"']{}
\end{tikzcd}.\end{center}
In writing $RO(C_7\rtimes C_3) = \bbZ\{1,\omega+\ol{\omega},\rho+\ol{\rho}\}\subset RU(C_7\rtimes C_3)$, the symbols $\omega+\ol{\omega}$ and $\rho+\ol{\rho}$ represent the real representations $\omega_0$ and $\rho_0$ underlying $\omega$ and $\rho$. For example, the identity $(\ol{\rho}-\rho) \cdot \ol{\rho} = \rho-(1+\omega+\ol{\omega})$ in $RU(C_7\rtimes C_3)$ gives a complex $t_{\rho-(\bbC+\omega+\ol{\omega})}$-element $J(\ol{\rho} \beta_{\rho}) \in \pi_{\rho-(\bbC+\omega+\ol{\omega})}^{C_7\rtimes C_3}C(a_\rho)$, with underlying real $t_{\rho_0-(2+2\omega_0)}$-element in $\pi_{\rho_0-(2+2\omega_0)}^{C_7\rtimes C_3}C(a_{\rho_0})$. The identity $a_\rho a_{\ol{\rho}} (\rho+\ol{\rho})\beta_\rho\beta_{\ol{\rho}} = (\rho+\ol{\rho})-(2+2(\omega+\ol{\omega}))$ then implies that this in fact lifts to $t\in \pi_{\rho_0-(2+2\omega_0)}^{C_7\rtimes C_3}C(a_{\rho_0}^2)$ satisfying $\res^{C_7\rtimes C_3}_e(\partial(t)) = 3j_{11}$.

\subsection{Example: the quaternion group of order \texorpdfstring{$8$}{8}}\label{ssec:quaternions}

We give an example with a symplectic irreducible. See also \cite{fujii1974on,mahammed1975ktheorie}. Let $G = Q_8$ be the quaternion group of order $8$. This sits in a short exact sequence 
\[
1\rightarrow C_2\rightarrow Q_8\rightarrow C_2\times C_2\rightarrow 1,
\]
and we have
\[
RU(Q_8) = \bbZ\{1,\sigma_1,\sigma_2,\sigma_3,H\},
\]
with $\bbZ\{1,\sigma_1,\sigma_2,\sigma_3\}\cong RO(C_2)$ consisting of those representations lifted from $C_2\times C_2$ and $H$ the tautological representation of $Q_8\subset Sp(1)\cong SU(2)$. Write $\bbH$ for the underlying $4$-dimensional real representation of $H$. As $Q_8$ acts freely on $S(\bbH)$, the $Q_8$-spectra $C(a_\bbH^m)$ admit all possible periodicities. Observe
\[
\sigma_i H = H,\qquad H^2 = 1+\sigma_1+\sigma_2+\sigma_3,\qquad e_H = 2 - H.
\]
This determines $\pi_{\ast\bbH}KU_{Q_8}$. Complex conjugation acts trivially, so the $E_2$-page of the HFPSS for $\pi_{\ast+\ast \bbH}KO_{Q_8}$ is given by
\[
H^\ast(C_2;\pi_{\ast+\ast \bbH}KU_{Q_8}) \cong \pi_{\ast H} KU_{Q_8} \otimes \bbZ[\beta^{\pm 2},\eta]/(2\eta).
\]
As $H$ is symplectic and the $\sigma_i$ are orthogonal we have 
\[
RO(Q_8) = \bbZ\{1,\sigma_1,\sigma_2,\sigma_3,2H\}\subset RU(Q_8),
\]
where ``$2H$'' represents the real represenation $\bbH$. This manifests in the HFPSS as a nontrivial differential $d_3(H) = H \beta^{-2} \eta^3$. As $\bbH$ is quaternionic of real dimension $4$, the class $\beta_{H}^2 \in \pi_{2
\bbH}KU_{Q_8}$ is a permanent cycle but $\beta_H \in \pi_{\bbH}KU_{Q_8}$ is not, and so $\pi_{\ast \bbH} KO_{Q_8}$ has the $2$-periodic pattern
\begin{center}\begin{tikzcd}[column sep=0.35cm]
\bbZ\{1,\sigma_1,\sigma_2,\sigma_3,2H\}&\ar[l,"a_\bbH"']\bbZ\{2,2\sigma_1,2\sigma_2,2\sigma_3,H\}\beta_H &\ar[l,"a_\bbH"']\bbZ\{1,\sigma_1,\sigma_2,\sigma_3,2H\}\beta_H^2&\ar[l,"a_\bbH"']{}
\end{tikzcd}.\end{center}
A calculation reveals that
\begin{align*}
2^{4m-1}(4-\bbH) &= a_\bbH^{2m}(2^{2m+1}-2^{2m-1}(1+\sigma_1+\sigma_2+\sigma_3)-2H)\beta_H^{2m},\\
2^{4m}(4-\bbH) &= a_\bbH^{2m+1}(2^{2m+1}-2^{2m-1}(1+\sigma_1+\sigma_2+\sigma_3)-H)\beta_H^{2m+1}
\end{align*}
for $m>0$: for example, the identities
\begin{align*}
(2-H)\cdot (1-\sigma_i) &= 2 (1-\sigma_i),\\
(2-H)\cdot (1+\sigma_1+\sigma_2+\sigma_3-2H) &= 8(1+\sigma_1+\sigma_2+\sigma_3-2H),\\
(2-H)\cdot (-H) &= \hphantom{8(}1+\sigma_1+\sigma_2+\sigma_3-2H
\end{align*}
together imply
\begin{align*}
a_\bbH^{2m}\cdot (2^{2m+1}-2^{2m-1}(1+\sigma_1+\sigma_2+\sigma_3))\beta_H^{2m} &= 2^{4m-1}(4-(1+\sigma_1+\sigma_2+\sigma_3)),\\
a_\bbH^{2m}\cdot (-2H)\beta_H^{2m} &= 2^{4m-1}(1+\sigma_1+\sigma_2+\sigma_3-\bbH),
\end{align*}
and adding these yields the first; the second is similar. Hence if we define
\[
p(n) = \begin{cases}
2n,&n\text{ even},\\
2n+1,&n\text{ odd},
\end{cases}
\]
then there are $t_{4-\bbH}$-elements
\[
u_{2^{p(n)}\bbH} \in \pi_{2^{p(n)}(4-\bbH)}^{Q_8}C(a_\bbH^{n+1})
\]
for $n\geq 0$, satisfying
\[
\res^{Q_8}_e(\partial(u_{2^{p(n)}}^k)) = \begin{cases}
k\cdot j_{4n+3},&n\text{ even},\\
k\cdot 4j_{4n+3},&n\text{ odd}
\end{cases}
\]
for $k\in \bbZ$. By identifying $S(n\bbH)$ as a $(4n-1)$-skeleton of $EQ_8$, this implies that if $R$ is a $Q_8$-ring spectrum then in the homotopy fixed point spectral sequence 
\[
E_2 = H^\ast(Q_8;\pi_\ast^e R[u_{\bbH}^{\pm 1},u_{\sigma_1}^{\pm 1},u_{\sigma_2}^{\pm 1},u_{\sigma_3}^{\pm 1}])\Rightarrow \pi_\star R_h^\wedge
\]
there are differentials
\[
d_{4(n+1)}(u_\bbH^{2^{p(n)}}) = \begin{cases}
j_{4n+3}\cdot a_\bbH^{n+1}u_\bbH^{-(n+1)} \cdot u_\bbH^{2^{p(n)}},&n\text{ even},\\
4j_{4n+3} \cdot a_\bbH^{n+1}u_\bbH^{-(n+1)} \cdot u_\bbH^{2^{p(n)}},&n\text{ odd}
\end{cases}
\]
for $n\geq 0$, where $a_\bbH$ is detected by the generator of $H^4(Q_8;\pi_{4-\bbH}^e S_{Q_8}) \cong \bbZ/(8)$. In other words,
\begin{gather*}
d_4(u_\bbH) = \nu a_\bbH,\quad d_4(u_\bbH^2) = 2\nu a_\bbH u_\bbH,\quad d_4(u_\bbH^4) = 4 \nu a_\bbH u_\bbH^3,\\
d_8(u_\bbH^8) = 4\sigma a_\bbH^2 u_\bbH^6,\\
d_{12}(u_\bbH^{16}) = j_{11}a_\bbH^3u_\bbH^{13},\quad d_{12}(u_\bbH^{32}) = 2j_{11}a_\bbH^3u_\bbH^{29},\quad d_{12}(u_\bbH^{64}) = 4j_{11}a_\bbH^3 u_\bbH^{61},\\
d_{16}(u_\bbH^{128}) = 4j_{15} a_\bbH^4u_\bbH^{124},
\end{gather*}
and so forth, up to orientation of $\nu$.

\subsection{Example: the binary octahedral group}

We give a larger symplectic example. Let $2O\subset Sp(1)$ denote the binary octahedral group, of order $48$. This group is of interest to chromatic homotopy theorists as the maximal subgroup $G_{48}\subset \bbG_2$ of the extended Morava stabilizer group associated to the Honda formal group law at the prime $2$ and height $2$. Consulting \cite{dokchitserxxxxgroupnames} we find that $2O$ has character table
\[
\begin{array}{c|cccccccc}
&1&4A&3&4B&2&8A&6&8B \\
\hline
1&1&1&1&1&1&1&1&1 \\
\rho_2&1&-1&1&1&1&-1&1&-1\\
\rho_3&2&0&-1&2&2&0&-1&0\\
\rho_4&2&0&-1&0&-2&-\sqrt{2}&1&\sqrt{2} \\
\rho_5&2&0&-1&0&-2&\sqrt{2}&1&-\sqrt{2} \\
\rho_6&3&1&0&-1&3&-1&0&-1\\
\rho_7&3&-1&0&-1&3&1&0&1\\
\rho_8&4&0&1&0&-4&0&-1&0
\end{array},
\]
where $\rho_4,\rho_5,\rho_8$ are symplectic and the rest are orthogonal. 

The tautological representation $\bbH$ of $2O\subset Sp(1)\cong SU(2)$ can be identified with $\rho_4$, with Euler class $e_{\rho_4} = 2-\rho_4$. The element $\alpha = 1+\rho_2-\rho_3+\rho_4+\rho_5-\rho_8 \in RU(2O)$ satisfies $e_{\rho_4}\cdot \alpha = \alpha$, and as $E2O = \colim_{n\rightarrow\infty}S(n\bbH)$ this produces a complex $t_\alpha$-element in $\pi_\alpha F(E2O_+,S_{2O})$.

Because $\bbH$ is real $4$-dimensional quaternionic, as with $G=Q_8$ we can compute
\begin{align*}
\pi_{2m\bbH}KO_{2O} &= \bbZ\{1,\rho_2,\rho_3,2\rho_4,2\rho_5,\rho_6,\rho_7,2\rho_8\}\beta_{\bbH}^{2m}\\
\pi_{(2m+1)\bbH}KO_{2O} &= \bbZ\{2,2\rho_2,2\rho_3, \rho_4,\rho_5,2\rho_6,2\rho_7,\rho_8\}\beta_{\bbH}^{2m+1},
\end{align*}
with the action of $a_\bbH$ determined by $a_\bbH \beta_{\bbH} = e_{\rho_4} = 2-\rho_4$. Hence for example
\[
48(4-\bbH) = a_\bbH^2 \cdot (64-8\rho_2-8\rho_3+17\cdot 2\rho_4-7\cdot 2\rho_5-16\rho_6+8\rho_7-6\cdot 2\rho_8)\beta_{\bbH}^2,
\]
yielding an invertible class $t \in \pi_{48(4-\bbH)}^{2O}C(a_\bbH^2)$ satisfying $\res^{2O}_e(\partial(t)) = \pm 8\sigma \in \pi_7 S$. If for example $u_{8\bbH} \in \pi_{8(4-\bbH)}^{Q_8}C(a_\bbH^2)$ is as in \cref{ssec:quaternions}, then $\res^{2O}_{Q_8}(t)\neq u_{8\bbH}^6$; instead $\epsilon = \res^{2O}_{Q_8}(t)\cdot u_{8\bbH}^{-6}\in \pi_0^{Q_8} C(a_\bbH^2)^\times$ is a unit satisfying $\res^{Q_8}_e(\partial(\epsilon)) = \pm 24\sigma$.

\begingroup
\raggedright
\bibliography{refs}
\bibliographystyle{alpha}
\endgroup

\end{document}